\documentclass[10pt]{article}

\usepackage{amsmath,amsfonts,amsthm,amssymb}
\usepackage{color}
\usepackage[]{graphics}
\usepackage{verbatim}
\usepackage{eepic,epic}
\usepackage{epsfig,subfigure,epstopdf}
\usepackage{url}
\allowdisplaybreaks

\usepackage[nohead,margin=1in]{geometry}

\newtheorem{theorem}{Theorem}[section]
\newtheorem{lemma}{Lemma}[section]

\newtheorem{remark}{Remark}[section]
\newtheorem{corollary}{Corollary}[section]
\numberwithin{equation}{section}

\newcommand{\al}{\alpha}

\renewcommand{\d}{{\rm d}}
\def\Dal{{\partial_t^\al}}

\def\Om{\Omega}
\def\II{(\Om)}
\def\bPtau{\bar\partial_\tau}

\begin{document}

\title{Incomplete Iterative Solution of the Subdiffusion Problem}

\author{Bangti Jin\thanks{Department of Computer Science, University College London, Gower Street, London, WC1E 2BT, UK.
(\texttt{b.jin@ucl.ac.uk, bangti.jin@gmail.com})}\and Zhi Zhou\thanks{Department of Applied Mathematics,
The Hong Kong Polytechnic University, Kowloon, Hong Kong (\texttt{zhizhou@polyu.edu.hk})}
}

\date{}

\maketitle

\begin{abstract}
In this work, we develop an efficient incomplete iterative scheme for the numerical solution of the subdiffusion
model involving a Caputo  derivative of order $\alpha\in(0,1)$ in time. It is based on
piecewise linear Galerkin finite element method in space and backward Euler convolution quadrature in time and solves one
linear algebraic system inexactly by an iterative algorithm at each time step. We present theoretical results for
both smooth and nonsmooth solutions, using novel weighted estimates of the time-stepping scheme. The analysis
indicates that with the number of iterations at each time level chosen properly, the error estimates are nearly
identical with that for the exact linear solver, and the theoretical findings provide guidelines
on the choice. Illustrative numerical results are presented to complement the theoretical analysis.

\end{abstract}

\section{Introduction}\label{sec:intro}
This work is concerned with efficient iterative solvers for the subdiffusion model. Let $\Omega\subset\mathbb{R}^d$
($d=1,2,3$) be a convex polyhedral domain with a boundary $\partial\Omega$. The subdiffusion model for the
function $u(t)$ reads:
\begin{equation}\label{eqn:fde}
   \left\{\begin{aligned}
     \Dal u(t) +A u(t) & = f(t),\quad \forall\, 0<t\leq T,\\
     u(0) & =v,\quad\text{in}\,\,\,\Omega,
   \end{aligned}\right.
\end{equation}
where $T>0$ is fixed, $f:(0,T)\rightarrow L^2(\Omega)$ and $v\in L^2(\Omega)$ are given functions, and $A=-\Delta:
D(A)\equiv H^1_0(\Omega)\cap H^2(\Omega)\rightarrow L^2(\Omega)$ denotes the negative Laplacian (with a zero
Dirichlet boundary condition). The notation $\Dal u$, $0<\al<1$, denotes the Caputo derivative of order $\al$ in
$t$, defined by \cite[p.\,91]{KilbasSrivastavaTrujillo:2006}
\begin{equation}\label{McT}
   \Dal u(t):= \frac{1}{\Gamma(1-\al)} \int_0^t(t-s)^{-\al}u'(s)\, {\rm d}s,
\end{equation}
where the Gamma function $\Gamma(\cdot)$ is defined by
$\Gamma(z):=\int_0^\infty s^{z-1}e^{-s}{\rm d}s$, $\Re z>0$.

The model \eqref{eqn:fde} describes so-called subdiffusion process, in which the mean squared displacement
of the particle grows only sublinearly with the time $t$, in contrast to the linear growth of Brownian
motion for normal diffusion. The sublinear growth captures important memory and hereditary effects of the
underlying physical process. Many experimental studies show that it can offer a superior
fit to experimental data than normal diffusion. The long list of applications includes thermal diffusion
in fractal domains, heat conduction with memory effect, and protein transport in cell membrane etc. We
refer interested readers to the reviews \cite{MetzlerKlafter:2000,MetzlerJeon:2014}
for physical background, mathematical modeling and long lists of applications.

Over the last two decades, a number of numerical methods have been developed for the model \eqref{eqn:fde},
e.g., finite element method, finite difference method and spectral method in space, and convolution
quadrature (CQ) and L1 type time-stepping schemes; See \cite{LinXu:2007,CuestaLubichPalencia:2006,SunWu:2006,McLeanMustapha:2009,JinLazarovZhou:SISC2016,Alikhanov:2015,MustaphaAbdallahFurati:2014,Stynes:2017,JinLiZhou:2019mc}
for a rather incomplete list, and \cite{JinLazarovZhou:2018review} for an overview on nonsmooth data
analysis, including optimal convergence rates.  The error analysis in all existing works requires the exact
resolution of resulting linear systems at each time step, which can be expensive. This motivates the use
of an iterative solver to approximately solve the resulting linear systems in order to enhance the
computational efficiency. There are many possible choices of iterative solvers, e.g., Krylov subspace methods,
multigrid methods, and domain decomposition methods.

In this work, we develop an efficient incomplete iterative scheme (IIS) for \eqref{eqn:fde}, based on
the Galerkin finite element method (FEM) in space, backward Euler CQ in time, and an iterative solver
for resulting linear  systems. We prove nearly optimal error estimates for both smooth and
nonsmooth solutions, under a contraction property of the iterative solver, cf. \eqref{eqn:Bhtau}, which
holds for many iterative methods. The IIS can maintain the overall accuracy if the number of iterations
at each time level is chosen suitably. Specifically, let $U_h^{n,M_n}$ be the solution by the IIS at $t_n$
obtained with $M_n$ iterations of the iterative solver, and $u$ the exact solution of \eqref{eqn:fde}.
Then for smooth solutions, e.g., $u\in C([0,T];D(A)) \cap C^2([0,T];H_0^1(\Omega))$, there exists a $\delta>0$ such that
\begin{equation*}
 \| U_h^{n,m}- u(t_n) \|_{L^2(\Omega)} \le c(u) (h^2 + \tau),\quad \text{for}~~ c_0 \kappa^m \le \delta,
\end{equation*}
where $c_0>0$ and $\kappa\in(0,1)$ are convergence parameters of the iterative method in a weighted energy norm;
see Theorem \ref{thm:iter-smooth}. That is, the number of iterations at each time level can be chosen uniformly
in time provided that it is large enough. In the absence of sufficient solution smoothness,
a uniform iteration number fails to ensure an optimal error estimate. The number of iterations at initial times
should be larger in order to compensate the singular behavior. For example,
for $v\in D(A)$ and $f\equiv0$,  there exists a $\delta>0$ such that
\begin{equation*}
 \| U_h^{n,M_n}- u(t_n) \|_{L^2(\Omega)} \le c \left(h^2+ \tau t_n^{\alpha-1}\ell_n\right)
 \| Av\|_{L^2(\Omega)}, 
 \end{equation*}
provided that $c_0\kappa^{M_n} \le \delta \ell_n^{-1}\min(t_n^\frac{\alpha}{2},1)$, with
$\ell_n=\ln(1+t_n/\tau)$. That is, it requires more iterations at starting time levels, even for
smooth initial data, which contrasts sharply with the standard parabolic counterpart \cite{BramblePasciakSammonThomee:1989}.
The proof relies crucially on certain new weighted estimates on the time stepping scheme, which
differ from known existing nonsmooth data error analysis \cite{JinLazarovZhou:SISC2016,Karaa:2018}.
The accuracy and efficiency of the scheme are illustrated by numerical experiments. The
numerical scheme and its rigorous error analysis for both smooth and nonsmooth solutions represent
the main contributions of this work.

The idea of incomplete iterations was first proposed for standard parabolic problems with smooth
solutions in \cite{DouglasDupontEwing:1979,BrambleSammon:1980}, and then extended in
\cite{Keeling:1989,BramblePasciakSammonThomee:1989,DuMing:2008} (including nonsmooth solutions);
see Chapter 11 of the monograph \cite{Thomee:2006} for a detailed discussion. Bramble et al
\cite{BramblePasciakSammonThomee:1989} proposed an incomplete iterative solver for a discrete
scheme based on Galerkin approximation in space and linear multistep backward difference in time,
and derived error estimates for nonsmooth initial data. Due to the nonlocality of the model
\eqref{eqn:fde} and limited smoothing properties, the analysis in these works does not
apply to problem \eqref{eqn:fde}.

The rest of the paper is organized as follows. In Section \ref{sec:prelim}, we describe the IIS.
Then in Sections \ref{sec:smooth} and \ref{sec:nonsmooth}, we analyze the scheme
for smooth and nonsmooth solutions, respectively. Finally, some numerical results are presented in
Section \ref{sec:numer} to complement the analysis. In two appendices, we collect useful basic estimates
and technical proofs. Throughout, the notation $c$ denotes a generic constant, which may differ at each
occurrence, but it is always independent of the time step size $\tau$ and mesh size $h$.

\section{The incomplete iterative scheme}\label{sec:prelim}
\subsection{Fully discrete scheme}
First, we describe a spatially semidiscrete scheme for problem \eqref{eqn:fde} based on the
Galerkin FEM. Let $\mathcal{T}_h$ be a shape regular quasi-uniform triangulation of the
domain $\Omega $ into $d$-simplexes, denoted by $T$, with a mesh size $h$. Over $\mathcal{T}_h$,
we define a continuous piecewise linear finite element space $X_h$ by
\begin{equation*}
  X_h= \left\{v_h\in H_0^1(\Omega):\ v_h|_T \mbox{ is a linear function},\ \forall\, T \in \mathcal{T}_h\right\}.
\end{equation*}
We define the $L^2(\Omega)$ projection $P_h:L^2(\Omega)\to X_h$ and
Ritz projection $R_h:H_0^1(\Omega)\to X_h$ by
\begin{equation*}
   \begin{aligned}
     (P_h \varphi,\chi) &=(\varphi,\chi) , &&\forall\, \chi\in X_h,\\
     (\nabla R_h\varphi,\nabla\chi) & = (\nabla \varphi,\nabla \chi),&& \forall\, \chi\in X_h,
   \end{aligned}
\end{equation*}
respectively, where $(\cdot,\cdot)$ denotes the $L^2(\Omega)$ inner product.

The semidiscrete Galerkin FEM for \eqref{eqn:fde} is to find $u_h(t)\in X_h$ such that
\begin{equation}\label{eqn:fem}
  (\Dal u_h,\chi) + (\nabla u_h,\nabla \chi) = (f,\chi),\quad\forall\, \chi\in X_h,\quad t>0,
\end{equation}
with $u_h(0)=v_h\in X_h$. Let $A_h: X_h\to X_h$ be the negative discrete Laplacian, i.e., $(A_h\varphi_h,\chi)
=(\nabla\varphi_h,\nabla\chi)$, for all $\varphi_h, \chi\in X_h.$ Then we rewrite \eqref{eqn:fem} as
\begin{equation}\label{eqn:semi}
  \Dal u_h(t)+A_h u_h(t) = f_h(t) , \quad \forall\, t>0,
\end{equation}
with $u_h(0)=v_h\in X_h$ and $f_h(t)=P_hf(t)$. The following identity holds
\begin{equation}\label{eqn:AR=PA}
  A_hR_h = P_h A.
\end{equation}

Next we partition the time interval $[0,T]$ uniformly, with grid points $t_n=n\tau$, $n=0,\ldots,N$, and
a time step size $\tau=T/N$. Recall the Riemann-Liouville derivative $^R\partial_t^\alpha\varphi
(t)=\frac{\rm d}{\rm d t}\frac{1}{\Gamma(1-\alpha)}\int_0^t(t-s)^{-\alpha}\varphi(s){\rm d} s$. The
backward Euler CQ for $^R\partial_t^\alpha
\varphi(t_n)$ is given by (with $\varphi^j=\varphi(t_j)$):
\begin{equation*}
  \bar\partial_\tau^\alpha \varphi^n = \tau^{-\alpha} \sum_{j=0}^nb_j^{(\alpha)}\varphi^{n-j},\quad\mbox{ with } (1-\xi)^\alpha=\sum_{j=0}^\infty b_j^{(\alpha)}\xi^j .
\end{equation*}
An estimate on $b_j^{(\alpha)}$ is given in Lemma \ref{lem:bdd-b} in Appendix A.
Since $\partial_t^\alpha\varphi= {^R\Dal}(\varphi(t)-\varphi(0))$ \cite[p. 91]{KilbasSrivastavaTrujillo:2006},
the fully discrete scheme for \eqref{eqn:fde} reads: Given $U_h^0=v_h\in X_h$, find $U_h^n\in X_h$ such that
\begin{align}\label{eqn:fully}
  \bar \partial_\tau^\alpha (U_h^n-U_h^0)+A_hU_h^n=f_h^n,\quad n=1,2,\ldots,N,
\end{align}
with $f_h^n=P_hf(t_n)$. The solution of \eqref{eqn:fully} can be represented by
\begin{equation}\label{eqn:sol-fully}
  U_h^n = F_{h,\tau}^n v_h + \tau \sum_{j=1}^n E_{h,\tau}^{n-j}f_h^j,
\end{equation}
where solution operators $F_{h,\tau}^n$ and $E_{h,\tau}^n$ are defined by
\begin{align*}
F_{h,\tau}^n &=\frac{1}{2\pi\mathrm{i}}
\int_{\Gamma_{\theta,\delta}^\tau } e^{z\tau(n-1)} \delta_\tau(e^{-z\tau})^{\alpha-1}({ \delta_\tau
(e^{-z\tau})^\alpha}+A_h)^{-1}\d z,\\  
E_{h,\tau}^n &= \frac{1}{2\pi\mathrm{i}}\int_{\Gamma_{\theta,
\delta}^\tau } e^{zn\tau} ({ \delta_\tau(e^{-z\tau})^\alpha}+A_h)^{-1}\,\d z,
\end{align*}
respectively, with
$\delta_\tau(\xi)=(1-\xi)/\tau$, $\Gamma_{\theta,\delta}^\tau :=\{ z\in \Gamma_{\theta,\delta}:
|\Im(z)|\le {\pi}/{\tau} \}$, and $\Gamma_{\theta,\delta}$ (oriented counterclockwise) defined by
(for $\theta\in(\frac{\pi}{2},\pi)$)
\begin{equation}\label{eqn:Gamma}
  \Gamma_{\theta,\delta}=\left\{z\in \mathbb{C}: |z|=\delta, |\arg z|\le \theta\right\}\cup
  \{z\in \mathbb{C}: z=\rho e^{\pm\mathrm{i}\theta}, \rho\ge \delta\} .
\end{equation}
Note that the formula for $F_\tau$ corrects a typo in \cite{JinLiZhou:2019mc}.

The solution operators $F_{h,\tau}^n$ and $E_{h,\tau}^n$ satisfy the following smoothing properties, where
$\|\cdot\|$ denotes the operator norm on $L^2(\Omega)$. The proof is standard, see, e.g.,
\cite{LubichSloanThomee:1996,JinLiZhou:2017}, and hence it is omitted.
\begin{lemma}\label{lem:stab-sol-op} For any $\beta\in[0,1]$, there hold
\begin{align*}
  \|A_h^\beta F_{h,\tau}^n\|\leq ct_{n+1}^{-\beta\alpha}, \quad
  \|A_h^\beta E_{h,\tau}^n\|\leq ct_{n+1}^{(1-\beta)\alpha-1}\quad\mbox{and}\quad \|A_h^\beta
  \bar\partial_\tau E_{h,\tau}^n\| \leq ct_{n+1}^{(1-\beta)\alpha-2}.
\end{align*}
\end{lemma}

\subsection{Incomplete iterative scheme (IIS)}\label{ssec:iter}
At each time level, the scheme \eqref{eqn:fully} requires solving a linear system. This can be
expensive for large-scale problems, e.g., three-dimensional problems. Hence, it is of much interest to develop efficient algorithms
that solve \eqref{eqn:fully} inexactly while maintaining the overall accuracy (in terms of convergence rate). In this work, we
propose an incomplete iterative BE scheme, by approximately solving the resulting linear systems.
Given $U_h^0$, $U_h^1,\ldots,U_h^{n-1}$, we use an iterative method to find an approximation to the solution
$\overline{U}_h^n$ of
\begin{equation}\label{eqn:lin}
  ( I+ \tau^\alpha A_h)\overline{U}_h^n = \tau^\alpha f_h^n - \sum_{j=1}^{n}b_j^{(\alpha)}U_h^{n-j} + \sum_{j=0}^{n}b_j^{(\alpha)} U_h^{0} ,
\end{equation}
with a starting guess $U_h^{n,0}$. Below we employ a second-order extrapolation:
\begin{equation}\label{eqn:Uhn0}
  U_h^{n,0}= 2U_h^{n-1}-U_h^{n-2},\quad n\geq 2.
\end{equation}
At time level $n$, an iterative method gives a sequence $U_h^{n,m}$  convergent to $\bar U_h^n$
as the iteration number $m\to \infty$. The IIS is given by setting
\begin{equation}\label{eqn:fully2}
    U_h^n = U_h^{n,M_n},
\end{equation}
for some parameter $M_n\in\mathbb{N}$, which may vary with $n$ and is to be specified.

The convergence analysis requires a certain contraction condition. We introduce a
weighted (energy like) norm $|\cdot|$ on the space $X_h$ defined by
\begin{equation}\label{eqn:norm-e}
    |\psi|=  \|  (I + \tau^\alpha A_h)^\frac{1}{2} \psi  \|_{L^2(\Omega)},\quad \forall \psi\in X_h.
\end{equation}
We assume that there exist $\kappa \in(0,1)$ and $c_0>0$:
\begin{equation}\label{eqn:Bhtau}
    |U_h^{n,m}-\overline{U}_h^n| \le c_0 \kappa^m |U_h^{n,0}-\overline{U}_h^n| \quad \text{for}~~m\ge1.
\end{equation}
The contraction property in the weighted norm $|\cdot|$ arises naturally in the study of many iterative solvers,
e.g., Krylov subspace methods \cite{Saad:2003}, multigrid methods \cite{Hackbusch:1985} and domain decomposition methods
\cite{ToselliWidlund:2005}. The constant $\kappa$ is related to the condition number of preconditioned
systems. The nonstandard norm $|\cdot|$ poses the main technical challenge in the analysis.

\section{Error analysis for smooth solutions}\label{sec:smooth}
Now we analyze the scheme \eqref{eqn:fully2} for smooth solutions, to give a first glance into its performance.
The more challenging case of nonsmooth solutions is deferred to Section \ref{sec:nonsmooth}. The analysis
below relies on two stability results on the time-stepping scheme \eqref{eqn:fully}. First, it satisfies the maximal $\ell^p$
regularity \cite[Theorem 5]{JinLiZhou:2018nm}. For any $1\leq p<\infty$, the norm $\|\cdot\|_{\ell^p(X)}$
of a sequence $(v_j)_{j=1}^n\subset X$ is defined by
$$\|(v_j)_{j=1}^n\|_{\ell^p(X)} = \big(\tau\sum_{j=1}^n\|v_j\|_{X}^p\big)^{1/p}.$$
\begin{lemma}\label{lem:max-lp}
For the solution $U_h^n$ of \eqref{eqn:fully} with $v_h=0$, there holds
\begin{equation*}
  \|(\bar\partial_\tau^\alpha U_h^j)_{j=1}^n\|_{\ell^p(L^2(\Omega))} + \|(A_h U_h^j)_{j=1}^n\|_{\ell^p(L^2(\Omega))} \leq c\|(f_h^j)_{j=1}^n\|_{\ell^p(L^2(\Omega))}, \ \ \forall 1<p<\infty.
\end{equation*}
\end{lemma}

The following stability estimate of the scheme \eqref{eqn:fully} is useful.
\begin{lemma}\label{lem:stab01}
Let $U_h^n$ be the  solution of \eqref{eqn:fully} with $v_h=0$. Then
\begin{align*}
    \|  U_h^n  \|_{L^2(\Omega)} + \| ( \nabla U_h^{j})_{j=1}^n \|_{\ell^q(L^2(\Omega))} & \le 
    c\| (A_h^{-\frac12}f_h^{j})_{j=1}^n \|_{\ell^q(L^2(\Omega))}, \ \ \forall q\in (\tfrac{2}{\alpha},\infty).
\end{align*}
\end{lemma}
\begin{proof}
By the representation \eqref{eqn:sol-fully}, we have
\begin{equation*}
  \begin{aligned}
    \|U_h^n\|_{L^2(\Omega)} & \leq \tau\sum_{j=1}^n\|E_{h,\tau}^{n-j}f_h^j\|_{L^2(\Omega)}
      \leq \tau\sum_{j=1}^n\|A_h^\frac12 E_{h,\tau}^{n-j}\|\|A_h^{-\frac12}f_h^j\|_{L^2(\Omega)}.
  \end{aligned}
\end{equation*}
Now for any $q>\frac{2}{\alpha}$, $(\frac\alpha2-1)\frac{q}{q-1}>-1$, and thus
$\tau\sum_{j=1}^n (t_{n+1}-t_j)^{(\frac\alpha2-1)\frac{q}{q-1}} <\infty$, cf. Lemma \ref{lem:basic-est1} in the appendix.
Next, by Lemma \ref{lem:stab-sol-op} and Young's inequality,
\begin{align*}
   \|U_h^n\|_{L^2(\Omega)} & \leq c\tau\sum_{j=1}^n(t_{n+1}-t_j)^{\frac{\alpha}{2}-1}\|A_h^{-\frac12}f_h^j\|_{L^2(\Omega)}\\
    & \leq c\|(A_h^{-\frac12}f_h^{j})_{j=1}^n \|_{\ell^q(L^2(\Omega))}<\infty.
\end{align*}
The bound on $\|(\nabla U_h^j)_{j=1}^n\|_{\ell^q(L^2(\Omega))}$ is due to Lemma \ref{lem:max-lp}. \qed
\end{proof}

Now we give an error estimate on the time-stepping scheme \eqref{eqn:fully} for smooth solutions, which serves
as a benchmark for the scheme \eqref{eqn:fully2}.
\begin{theorem}\label{thm:error-smooth-sol}
Let $u$ be the solution to \eqref{eqn:fde}, and $U_h^n$ be the solution of \eqref{eqn:fully} with
$v_h= R_h v$. If $u\in C^2([0,T];H_0^1(\Omega))\cap C^1([0,T];D(A))$, then
\begin{equation*}
 \| U_h^n- u(t_n) \|_{L^2(\Omega)} \le c(u) (h^2 + \tau).
\end{equation*}
\end{theorem}
\begin{proof}
In a customary way, we split the error $e^n\equiv U_h^n-u(t_n)$ into
\begin{equation*}
    e^n = (U_h^n-R_h u(t_n)) + (R_h u(t_n)-u(t_n)) =: \vartheta^n + \varrho^n.
\end{equation*}
It suffices to bound the terms $\varrho^n$ and $\vartheta^n$. Clearly,
\begin{equation}\label{eqn:rho}
  \|  \varrho^n \|_{L^2(\Omega)} \le ch^2\|u\|_{C([0,T];H^2(\Omega))}.
\end{equation}
It remains to bound $\vartheta^n$. Note that $\vartheta^n$ satisfies $\vartheta^0=0$ and
\begin{align*}
  \bPtau^\alpha \vartheta^n + A_h \vartheta^n & = \bPtau^\al(U_h^n-R_hu(t_n)) + A_h(U_h^n-R_hu(t_n)) \\
     & = \big(\bPtau^\alpha (U_h^n-v_h)+A_hU_h^n\big) - \big(\bPtau^\alpha R_h(u(t_n)-v_h) + A_hR_hu(t_n)\big).
\end{align*}
It follows from the identity \eqref{eqn:AR=PA}, and equations \eqref{eqn:fully} and \eqref{eqn:fde} that
\begin{align*}
  \bPtau^\al \vartheta^n + A_h \vartheta^n & = - \bPtau^\alpha R_h(u(t_n)-v_h) + P_h\partial_t^\alpha(u(t_n)-v)\\
    & = ( P_h - R_h ) \Dal u(t_n) - R_h (\bPtau^\alpha - \Dal )(u(t_n)-v).
\end{align*}
Since the solution $u$ is smooth, by the approximation properties of $R_h$ and $P_h$,
\begin{align}\label{eqn:approx1}
 \|(P_h-R_h)\Dal u(t_n)\|_{L^2(\Omega)} 
  &\le c h^2\|u\|_{C^1([0,T];D(A))},
\end{align}
and further, by the approximation property of $\bar\partial_\tau^\alpha$ to $^R\partial_t^\alpha$ \cite{Lubich:1986} 
\begin{equation}\label{eqn:approx2}
  \begin{aligned}
   \| R_h (\bPtau^\al - \Dal )(u(t_n)-v)\|_{L^2(\Omega)} & \le \|  (\bPtau^\al - \Dal )(u(t_n)-v)\|_{H_0^1(\Omega)}\\
   & \le c \tau\|u\|_{C^2([0,T];H_0^1(\Omega))}.
  \end{aligned}
\end{equation}
Now since $\vartheta^0=0$, the estimate follows from Lemma \ref{lem:stab01}. \qed
\end{proof}

Next we can state the main result of this part, i.e., convergence rate of the scheme
\eqref{eqn:fully2} for smooth solutions: it can achieve the accuracy of \eqref{eqn:fully},
if a large enough but fixed number $m$ of iterations is taken at each time level. In the
proof, we denote the space $X_h$ equipped with the norm $|\cdot|$ defined in \eqref{eqn:norm-e} by $X_{h,\tau}$.

\begin{theorem}\label{thm:iter-smooth}
Let $u$ and  $U_h^n\equiv U_h^{n,m}$ be the solutions of \eqref{eqn:fde} and \eqref{eqn:Uhn0}-\eqref{eqn:fully2} with
$v_h = R_h v$, respectively,  and let $U_h^1= \overline{U}_h^1$. If $u\in C^2([0,T];H_0^1(\Omega))\cap C^1([0,T];D(A))$,
then there exists a $\delta>0$ such that
\begin{equation*}
 \| U_h^{n}- u(t_n) \|_{L^2(\Omega)} \le c(u) (h^2 + \tau),\quad \text{for}~~ c_0 \kappa^m \le \delta.
\end{equation*}
\end{theorem}
\vspace{-0.4cm}
\begin{proof}
In a customary way, we split the error $e^{n,m}=U_h^{n,m}-u(t_n)$ into
\begin{equation*}
    e^{n,m}= (U_h^{n,m}-R_h u(t_n)) + (R_h u(t_n)-u(t_n)) =: \vartheta^n + \varrho^n.
\end{equation*}
In view of the estimate \eqref{eqn:rho}, it suffices to bound $\vartheta^n$. We break the lengthy and technical
proof into three steps.

\noindent{\bf Step 1}: Bound $\vartheta^n$ by local truncation errors.
Note that $\vartheta^n$ satisfies $\vartheta^0=0$ and for $n=1,\ldots,N$
\begin{align*}
    \bPtau^\alpha \vartheta^n + A_h\vartheta^n
      = & \big(\bPtau^\alpha (U_h^{n,m}-v_h)+ A_hU_h^{n,m}\big)\\
       & -\big(\bPtau^\alpha (R_hu(t_n)-v_h) + A_hR_hu(t_n)\big).
\end{align*}
Let the auxiliary function $\overline{U}_h^n\in X_h$ satisfy $\overline{U}_h^0 = R_hv$ and
\begin{equation*}
  \tau^{-\alpha} \Big(\overline{U}_h^n + \sum_{j=1}^nb_{j}^{(\alpha)} U_h^{n-j,m} -  \sum_{j=0}^n b_{j}^{(\alpha)} U_h^0\Big) + A_h \overline{U}_h^n  = f_h^n,\quad n=1,2,\ldots,N.
\end{equation*}
Therefore, there holds
\begin{align*}
  \bar\partial_\tau^\alpha(U_h^{n,m}-v_h) +A_hU_h^{n,m} =& P_h[\partial_t^\alpha(u(t_n)-v)+Au(t_n)] \\
      &+\tau^{-\alpha}(U_h^{n,m}-\overline{U}_h^n) +A_h(U_h^{n,m}-\overline{U}_h^n).
\end{align*}
This and the identities \eqref{eqn:AR=PA}, \eqref{eqn:fde} and \eqref{eqn:fully2} imply
\begin{equation}\label{eqn:sig1}
    \bPtau^\al \vartheta^n + A_h \vartheta^n= \sigma^n,\quad \mbox{with } \sigma^n = (I+\tau^\alpha A_h) \eta^n  + \omega^n,
\end{equation}
with the errors $\eta^n$ and $\omega^n$ given by
\begin{align*}
  \eta^n &=\tau^{-\alpha} (U_h^{n,m} - \overline{U}_h^n),\\
  \omega^n &= ( P_h - R_h ) \Dal (u(t_n)-v) - R_h (\bPtau^\alpha - \Dal)(u(t_n)-v).
\end{align*}
By Lemma \ref{lem:stab01} and triangle inequality, for any $q\in(\frac2\alpha,\infty)$ and $n=1,2,...,N$,
\begin{align*}
    \|\vartheta^n \|_{L^2(\Omega)}  &\le  c\| (A_h^{-\frac12}\sigma^j)_{j=1}^n\|_{\ell^q(L^2(\Omega))}\\
   &\leq  c\|((I+\tau^{\alpha} A_h)A_h^{-\frac12}\eta^{j})_{j=1}^n\|_{\ell^q(L^2(\Omega))}
  + c\|(A_h^{-\frac12}\omega^j)_{j=1}^n\|_{\ell^q(L^2(\Omega))}.
\end{align*}
Since $u\in C^2([0,T];H_0^1(\Omega))\cap C^1([0,T];D(A))$, \eqref{eqn:approx1} and \eqref{eqn:approx2} imply
\begin{align*}
 \| A_h^{-\frac12}\omega^j\|_{L^2(\Omega)}\leq c\|\omega^j\|_{L^2(\Omega)}\leq c(u)(h^2+\tau).
\end{align*}
Further, since $(I+\tau^\alpha A_h)A_h^{-\frac{1}{2}}=(A_h^{-1}+\tau^\alpha I)^\frac12(I+\tau^\alpha A_h)^\frac{1}{2}$, we have
\begin{equation}\label{eqn:eta}
    \|(I+\tau^\alpha A_h)A_h^{-\frac12}\eta^j\|_{L^2(\Omega)}  \leq  c|\eta^j|.
\end{equation}
The last three estimates imply
\begin{align}\label{est1}
 \| \vartheta^n  \|_{L^2(\Omega)}  \leq  c \|(\eta^{j})_{j=1}^n\|_{\ell^q(X_{h,\tau})}  +  c(u) (h^2 + \tau).
\end{align}

\noindent{\bf Step 2}: Bound the summand $|\eta^j|$. Given a tolerance $\delta>0$ to be determined,
under assumption \eqref{eqn:Bhtau}, there exists an integer $m\in\mathbb{N}$ such that $c_0 \kappa^m \le \delta$
and by triangle inequality,
\begin{equation*}
 |U_h^{n,m} - \overline{U}_h^n| \le \delta |U_h^{n,0}-\overline{U}_h^n| \leq \delta \big( |U_h^{n,0} - U_h^{n,m}| + |U_h^{n,m} - \overline{U}_h^n | \big).
\end{equation*}
With $\epsilon=\delta(1-\delta)^{-1}$, rearranging the inequality gives
\begin{equation*}
|U_h^{n,m} - \overline{U}_h^n|  \le \epsilon |U_h^{n,0} - U_h^{n,m} |.
\end{equation*}
Hence,
\begin{equation*}
 |\eta^n| = \tau^{-\alpha} |U_h^{n,m} - \overline{U}_h^n|
 \le \epsilon \tau^{-\alpha} |U_h^{n,0} - U_h^{n,m} |.
\end{equation*}
Meanwhile, the choice of $U_h^{n,0}$ in \eqref{eqn:Uhn0} implies
\begin{align*}
 U_h^{n,m} - U_h^{n,0} & = U_h^{n,m} - 2U_h^{n-1} + U_h^{n-2}
  = \tau  (\bar \partial_\tau  U_h^{n,m}  - \bar \partial_\tau  U_h^{n-1}) \\
 &=\tau  \bar \partial_\tau  \vartheta^n - \tau  \bar \partial_\tau  \vartheta^{n-1} +  \tau^{2}\bar \partial_\tau^2 R_h u(t_n) .
\end{align*}
The last two estimates together imply
\begin{equation}\label{est2}
\begin{split}
  |\eta^n| &\le c\epsilon\tau^{1-\alpha}(|\bar\partial_\tau\vartheta^n|+|\bar\partial_\tau\vartheta^{n-1}|) + c\epsilon \tau^{2-\alpha} | R_h\bar\partial_\tau^{2} u(t_n) |\\
  &\le c\epsilon \tau^{1-\alpha} (|\bar \partial_\tau\vartheta^n|+|\bar \partial_\tau  \vartheta^{n-1}|) + c\epsilon \tau^{2-\alpha}\|u\|_{C^2([0,T];H_0^1(\Omega))}.
  \end{split}
\end{equation}
This, \eqref{est1} and the standard inverse inequality in time yield
\begin{equation}\label{est3}
\begin{split}
 \| \vartheta^n  \|_{L^2(\Omega)} &\le c\epsilon\tau^{1-\alpha}\|(\bar \partial_\tau\vartheta^j)_{j=1}^n\|_{\ell^q(X_{h,\tau})} +  c(u) (h^2 + \tau)\\
 &\le c\epsilon\|(\bar \partial_\tau^\alpha\vartheta^j)_{j=1}^n\|_{\ell^q(X_{h,\tau})} +  c(u) (h^2 + \tau).
  \end{split}
\end{equation}

\noindent{\bf Step 3}: Bound $\|\vartheta^n\|_{L^2(\Omega)}$ explicitly. Let $I_h=(I+\tau^{\alpha} A_h)^{-\frac12}$.
Then the identity $|\bar \partial_\tau^\alpha \vartheta^j|= \|(I+\tau^{\alpha} A_h)  \bar \partial_\tau^\alpha
I_h\vartheta^j\|_{L^2(\Omega)}$ and the triangle inequality imply
\begin{align*}
 \|(\bar\partial_\tau^\alpha \vartheta^j)_{j=1}^n\|_{\ell^q(X_{h,\tau})}
  & \le  \|( \bar \partial_\tau^\alpha I_h\vartheta^j)_{j=1}^n\|_{\ell^q(L^2(\Omega))}
+   \tau^\alpha \|(\bar \partial_\tau^\alpha A_hI_h\vartheta^j)_{j=1}^n\|_{\ell^q(L^2(\Omega))}\\
  &:={\rm I}+{\rm II} .
\end{align*}
By Lemma \ref{lem:max-lp}, we have
\begin{align*}
  {\rm I}&\leq c\|(I_h\sigma^j)_{j=1}^n\|_{\ell^q(L^2(\Omega))},
\end{align*}
and similarly, the inverse inequality (in time) and Lemma \ref{lem:max-lp} yield
\begin{align*}
 {\rm II} & \leq  c\| (A_hI_h\vartheta^j)_{j=1}^n\|_{\ell^q(L^2(\Omega))} \le c\|(I_h\sigma^j)_{j=1}^n\|_{\ell^q(L^2(\Omega))}.
\end{align*}
Combining the last three estimates with \eqref{eqn:approx1}--\eqref{eqn:sig1} gives
\begin{equation}\label{est4}
\begin{split}
\|(\bar \partial_\tau^\alpha\vartheta^j)_{j=1}^n\|_{\ell^q(X_{h,\tau})}
 &\le c\|(I_h\sigma^j)_{j=1}^n\|_{\ell^q(L^2(\Omega))}\\
 &\le c(u)(\tau+h^2)+c\|(\eta^j)_{j=1}^n\|_{\ell^q(X_{h,\tau})}.
\end{split}
\end{equation}
Now it follows from \eqref{est2} and \eqref{est4} that
\begin{equation*}
 \|(\bar\partial_\tau^\alpha\vartheta^j)_{j=1}^n\|_{\ell^q(X_{h,\tau})} \le c(u) (\tau+h^2) + c\epsilon \|(\bar \partial_\tau^\alpha \vartheta^j)_{j=1}^n\|_{\ell^q(X_{h,\tau})}.
\end{equation*}
Thus by choosing a sufficiently small $\epsilon$, we get
\begin{equation*}
 \|(\bar \partial_\tau^\alpha \vartheta^j)_{j=1}^n\|_{\ell^q(X_{h,\tau})} \le  c(u) (\tau+h^2).
\end{equation*}
This and \eqref{est3} give $\|  \vartheta^n \|_{L^2(\Omega)} \le  c(u) (\tau+h^2)$,
which completes the proof.\qed
\end{proof}

\begin{remark}\label{rmk:smooth-sol}
The regularity requirement $u\in C^1([0,T];D(A)) \cap C^2([0,T];H_0^1(\Omega))$ is restrictive for the subdiffusion
model \eqref{eqn:fde}, due to the well known limited smoothing properties of the corresponding solution operators.
It holds only under certain compatibility conditions on the initial data $v$ and the source term $f$. It holds if
$v=0$, $f(0)=f'(0)=0$ and $f''\in L^\infty(0,T; H^\epsilon(\Omega))$ with a small $\epsilon>0$. The proof uses
crucially the maximal $\ell^p$ regularity estimate, which differs greatly from the argument for the case of nonsmooth
solutions below and also the argument for the standard parabolic equation.
\end{remark}

\section{Error analysis for nonsmooth solutions}\label{sec:nonsmooth}
Now we analyze the case that the solution $u$ is nonsmooth, and derive error estimates nearly optimal
with respect to data regularity. Nonsmooth solutions are characteristic of problem \eqref{eqn:fde}:
with $f=0$ and $A^\beta v\in L^2(\Omega)$, $\beta\in[0,1]$, $u(t)$ satisfies
\cite[Theorem 2.1]{JinLazarovZhou:2018review}
\begin{equation*}
\| \partial_t^k u(t)  \|_{L^2(\Omega)} \le c t^{\beta\alpha-k}\| A^\beta v  \|_{L^2(\Omega)}.
\end{equation*}
Thus, it is important to analyze numerical methods for nonsmooth solutions. To this end, we split the error
$\|U_h^{n,M_n}-u(t_n)\|_{L^2(\Omega)}$ into
$$U_h^{n,M_n} - u(t_n) = (U_h^{n,M_n} - u_h(t_n)) + (u_h(t_n) - u(t_n)),$$
and the spatial error $\|u(t)-u_h(t)\|_{L^2(\Omega)}$ satisfies (with $\ell_h=\ln(1/h+1)$) \cite{JinLazarovZhou:2018review}
\begin{equation*}
   \|(u-u_h)(t)\|_{L^2(\Omega)} \le \left\{\begin{array}{ll}
      ch^2 \| A v\|_{L^2(\Omega)}, &\quad \mbox{if } v_h=R_hv,\\
      ch^2\ell_h t^{-\alpha} \|v\|_{L^2(\Omega)}, & \quad \mbox{if } v_h=P_hv.
   \end{array}\right.
\end{equation*}
Thus, we focus on the temporal error $\|U_h^{n,M_n} - u_h(t_n)\|_{L^2(\Omega)}$.  The analysis below uses
certain \textit{a priori} estimates on the semidiscrete solutions $u_h$ and its fully discrete approximations
$\bar\partial_\tau^\alpha u_h(t_n)$. The proofs follow the standard (discrete) Laplace transform
techniques and thus are deferred to Appendix \ref{app:reg}.

\begin{lemma}\label{lem:est-uh-weight}
Let $u_h$ be the solution to \eqref{eqn:semi} with $f=0$. Then for $\beta\in[0,1]$
\begin{equation*}
  |\bar \partial_\tau^2u_h(t_n)|\leq ct_n^{\beta\alpha-2}\|A_h^\beta v_h\|_{L^2(\Omega)},\quad n> 2.
\end{equation*}
\end{lemma}

\begin{lemma}\label{lem:tal2}
Let $u_h(t)$ be the solution to \eqref{eqn:semi} with $f=0$ and $y_h(t) = u_h(t) - v_h$. Then for any $\beta\in[0,1]$, the following statements hold.
\begin{itemize}
\item[$\rm(i)$] If $Av\in L^2(\Omega)$ and $v_h=R_hv$, then
\begin{equation*}
   \| A_h^{\beta}\big(\partial_t^{\alpha } y_h(t_n) - \bar \partial_\tau^{\alpha} y_h(t_n)\big)\|_{L^2(\Omega)} \le c \tau  t_n^{-1-\beta\alpha} \|Av\|_{L^2(\Omega)}.
\end{equation*}
\item[$\rm(ii)$] If $v\in L^2(\Omega)$ and $v_h=P_hv$, then
\begin{equation*}
\| A_h^{-\beta}\big(\partial_t^{\alpha } y_h(t_n) - \bar \partial_\tau^{\alpha} y_h(t_n)\big)\|_{L^2(\Omega)} \le c \tau  t_n^{-1-(1-\beta)\alpha} \|v\|_{L^2(\Omega)}.
\end{equation*}
\end{itemize}
\end{lemma}

\begin{corollary}\label{cor:tal3}
Let $u_h(t)$ be the solution to \eqref{eqn:semi} with $f\equiv0$ and $y_h(t) = u_h(t) - v_h$. If $v\in L^2(\Omega)$ and $v_h=P_hv$, then for any $\beta\in[0,1]$,
\begin{equation*}
\| A_h^{-\beta} \bar\partial_\tau\big(\partial_t^\alpha  y_h(t_n) - \bar \partial_\tau^\alpha  y_h(t_n)\big)\|_{L^2(\Omega)}
\le c \tau  t_n^{-2-(1-\beta)\alpha} \|v\|_{L^2(\Omega)}.
\end{equation*}
\end{corollary}

Below we analyze the homogeneous problem with the smooth and nonsmooth initial
data separately, since the requisite estimates differ substantially. The main
results of this section, i.e., error estimates for the incomplete iterative scheme \eqref{eqn:fully2}
are given in Theorems \ref{thm:err-smooth-ini} and \ref{thm:err-nonsmooth-ini}.
\subsection{Smooth initial data}
First, we analyze the case of smooth initial data, i.e., $Av\in L^2(\Omega)$. We begin with a simple weighted
estimate of inverse inequality type. The shorthand LHS denotes the left hand side.
\begin{lemma}\label{lem:inverse}
For any $\varphi^{j}\in X_h$ {\rm(}with $\varphi^0=0${\rm)}, and $\gamma\in(0,1)$, there holds
\begin{equation*}
\tau  \sum_{j=1}^n (t_{n+1}-t_j)^{\frac\alpha2-1}  \|  \bar \partial_\tau^\gamma \varphi^j \|_{L^2(\Omega)}
\le c\tau^{1-\gamma}   \sum_{j=1}^n (t_{n+1}-t_j)^{\frac\alpha2-1}\| \varphi^j\|_{L^2(\Omega)}.
\end{equation*}
\end{lemma}
\begin{proof}
Since $\varphi^0=0$, Lemma \ref{lem:bdd-b} and changing the summation order yield
\begin{align*}
  {\rm LHS} &\leq \tau^{1-\gamma} \sum_{j=1}^n (t_{n+1}-t_j)^{\frac\alpha2-1} \sum_{\ell=0}^j |b_{j-\ell}^{(\gamma)}| \| \varphi^\ell\|_{L^2(\Omega)}\\
 & \le c\tau^{\frac\alpha2-\gamma}\sum_{\ell=1}^n \| \varphi^\ell\|_{L^2(\Omega)} \sum_{i=0}^{n-\ell} ( {n-\ell}+1- i)^{\frac\alpha2-1} (i+1)^{-\gamma-1}.
\end{align*}
The desired assertion follows directly from Lemma \ref{lem:basic-est1}.\qed
\end{proof}

The next result gives a weighted estimate on the time stepping scheme \eqref{eqn:fully}.
\begin{lemma}\label{lem:disc-stab-2}
Let $e^n\in X_h$ satisfy $e^0=0$ and
\begin{equation*}
   \bar\partial_\tau^\alpha e^n + A_h e^n = \sigma^n,\quad n=1,\ldots,N.
\end{equation*}
Then with $\ell_n=\ln(1+t_n/\tau)$, there holds
\begin{equation*}
\tau \sum_{j=1}^n (t_{n+1}-t_j)^{\frac\alpha2-1}  |\bar \partial_\tau^\alpha e^j|
\le c\tau \ell_n  \sum_{j=1}^n (t_{n+1}-t_j)^{\frac\alpha2-1}\|(I+\tau^\alpha A_h)^{-\frac12}\sigma^j \|_{L^2(\Omega)}.
\end{equation*}
\end{lemma}
\vspace{-0.2cm}
\begin{proof}
Let $I_h=(I+\tau^\alpha A_h)^{-\frac12}$. By the identity $ \bar\partial_\tau^\alpha e^j =\sigma^j- A_h e^j $, we have
\begin{align*}
| \bar \partial_\tau^\alpha e^j |&
\leq \|I_h \bar \partial_\tau^\alpha e^j \|_{L^2(\Omega)} + \tau^{\alpha} \|I_hA_h\bar\partial_\tau^\alpha e^j\|_{L^2(\Omega)}\\
&\leq \|I_hA_he^j \|_{L^2(\Omega)} + \|I_h\sigma^j\|_{L^2(\Omega)} + \tau^\alpha \|I_hA_h\bar\partial_\tau^\alpha e^j\|_{L^2(\Omega)}.
\end{align*}
Then the inverse estimate in Lemma \ref{lem:inverse} implies
\begin{align*}
  {\rm LHS}\leq &\tau \sum_{j=1}^n (t_{n+1}-t_j)^{\frac\alpha2-1}(\|I_hA_he^j\|_{L^2(\Omega)}+\|I_h\sigma^j\|_{L^2(\Omega)}) \\
   &\quad + \tau^{1+\alpha} \sum_{j=1}^n (t_{n+1}-t_j)^{\frac\alpha2-1}  \|I_h A_h\bar \partial_\tau^\alpha e^j \|_{L^2(\Omega)}\\
\le& c\tau\sum_{j=1}^n (t_{n+1}-t_j)^{\frac\alpha2-1}(\|I_hA_he^j\|_{L^2(\Omega)}+\|I_h\sigma^j\|_{L^2(\Omega)}).
\end{align*}
Now the representation $e^j=\tau\sum_{\ell=1}^jE_{h,\tau}^{j-\ell}\sigma^\ell$ in \eqref{eqn:sol-fully}, and Lemma \ref{lem:stab-sol-op} yield
\begin{align*}
  \|I_hA_he^j\|_{L^2(\Omega)} \le c\tau \sum_{\ell=1}^j(t_{j+1}-t_\ell)^{-1}\|I_h\sigma^\ell\|_{L^2(\Omega)}.
\end{align*}
The last two estimates and changing the summation order give
\begin{align*}
{\rm LHS}\leq &c\tau\sum_{j=1}^n(t_{n+1}-t_j)^{\frac\alpha2-1}\Big(\tau\sum_{\ell=1}^j(t_{j+1}-t_\ell)^{-1}
\|I_h\sigma^\ell\|_{L^2(\Omega)}+\|I_h\sigma^j\|_{L^2(\Omega)}\Big)\\
= &c\tau\sum_{\ell=1}^n\|I_h\sigma^\ell\|_{L^2(\Omega)} \Big(\tau \sum_{j=\ell}^n(t_{n+1}-t_j)^{\frac\alpha2-1}(t_{j+1}-t_\ell)^{-1}
+(t_{n+1}-t_\ell)^{\frac\alpha2-1}\Big).
\end{align*}
This and Lemma \ref{lem:basic-est1} complete the proof.\qed
\end{proof}

Now we can give an error estimate for the scheme \eqref{eqn:fully2} for smooth initial data, i.e., $Av\in L^2(\Omega)$. The error bound
for \eqref{eqn:fully2} is identical with that for the exact linear solver, up to the logarithmic factor $\ell_n$.
\begin{theorem}\label{thm:err-smooth-ini}
Let $A v\in L^2(\Omega)$ and condition \eqref{eqn:Bhtau} hold. Let $U_h^n\equiv U_h^{n,M_n}$
be the solution of \eqref{eqn:Uhn0}--\eqref{eqn:fully2} with $f=0$ and $v_h=R_hv$, and let
$U_h^n= \overline{U}_h^n$ for $n=1,2$. Then with $\ell_n=\ln(1+t_n/\tau)$, there exists a $\delta>0$ such that
\begin{equation*}
   \|U_h^n-u_h(t_n)\|_{L^2(\Omega)} \le c\tau t_n^{\alpha-1}\ell_n\| Av\|_{L^2(\Omega)},\quad \text{if}~c_0\kappa^{M_n} \le \delta \ell_n^{-1}\min(t_n^\frac{\alpha}{2},1).
\end{equation*}
\end{theorem}
\vspace{-0.3cm}
\begin{proof}
The desired estimate holds trivially for $n=1,2$, and thus we consider only $n>2$.
Note that $e^n =U_h^n-u_h(t_n)$ satisfies $e^0=0$ and
\begin{equation}\label{eqn:err-smooth-data}
 \bar\partial_\tau^\alpha e^n + A_h e^n = \sigma^n :=\omega^n +  (I+\tau^\alpha A_h){\eta}^n,
\end{equation}
where $\omega^n$ and $\eta^n$ are defined respectively by
\begin{equation}\label{eqn:omega-eta}
 \omega^n = -(\bar\partial_\tau^\alpha-\partial_t^\alpha)(u_h(t_n)-v_h) \qquad\text{and}\qquad {\eta}^n =\tau^{-\alpha} (U_h^n - \overline{U}_h^n),
\end{equation}
where the auxiliary function $\overline{U}_h^n\in X_h$ satisfies $\overline{U}_h^0 = R_hv$ and
\begin{equation}\label{eqn:barUhn}
  \tau^{-\alpha} \Big(\overline{U}_h^n + \sum_{j=1}^n {b_{j}^{(\alpha)} U_h^{n-j}} -  \sum_{j=0}^n b_{j}^{(\alpha)} U_h^0\Big) + A_h \overline{U}_h^n  = f_h^n,\quad n=1,2,\ldots,N.
\end{equation}
The rest of the proof consists of three steps.

\noindent{\bf Step 1}: Bound $\|e^n\|_{L^2(\Omega)}$ by local truncation errors. Since $e^0=0$, by the error equation
\eqref{eqn:err-smooth-data}, \eqref{eqn:sol-fully} and Lemma \ref{lem:stab-sol-op}, $e^n$ is bounded by
\begin{align*}
     \|  e^n  \|_{L^2(\Omega)}  \le &  c \tau \sum_{j=1}^n (t_{n+1}-t_j)^{\alpha-1}\|\omega^j \|_{L^2(\Omega)} \\
      &+ c \tau \sum_{j=1}^n (t_{n+1}-t_j)^{\frac\alpha2-1}\| A_h^{-\frac12}(I+\tau^\alpha A_h){\eta}^j\|_{L^2(\Omega)}: = {\rm I} + {\rm II}.
\end{align*}
It suffices to bound the two terms ${\rm I}$ and ${\rm II}$. By Lemmas \ref{lem:tal2}(i) and
\ref{lem:basic-est1}, the first term ${\rm I}$ can be bounded by
\begin{align}
  {\rm I} & \leq c\tau^2 \|Av\|_{L^2(\Omega)} \sum_{j=1}^n (t_{n+1}-t_j)^{\alpha-1}t_j^{-1}
\leq c\tau t_n^{\alpha-1}\ell_n \|Av\|_{L^2(\Omega)}.\label{eqn:est-om}
\end{align}
Further, it follows directly from \eqref{eqn:eta} that
\begin{equation}\label{eqn:est-eta}
  {\rm II} \leq c\tau \sum_{j=1}^n(t_{n+1}-t_j)^{\frac{\alpha}{2}-1}|\eta^j|.
\end{equation}

\noindent {\bf Step 2}: Bound the summand $|\eta^j|$. By assumption \eqref{eqn:Bhtau} and triangle inequality,
for any integer $M_n$, there holds
\begin{equation*}
 |U_h^{n} - \overline{U}_h^n| \le c_0 \kappa^{M_n} \big( |U_h^{n,0} - U_h^{n}| + |U_h^{n}-\overline{U}_h^n | \big).
\end{equation*}
Now choose $M_n$ such that $c_0\kappa^{M_n} \le \delta \min(t_n^{\frac\alpha2},1)\ell_n^{-1}$, and let $\epsilon=\frac{\delta}{1-\delta}\ell_n^{-1}$. Since
$c_0 \kappa^{M_n}/(1-c_0 \kappa^{M_n})\leq \epsilon t_n^\frac{\alpha}{2}$, rearranging the terms yields
\begin{equation*}
 |U_h^{n} - \overline{U}_h^n|\le \epsilon t_n^{\frac\alpha2} |U_h^{n,0} - U_h^{n} |,
\end{equation*}
and by the definition of $\eta^n$ in \eqref{eqn:omega-eta}, $\eta^1=\eta^2=0$ and for $n>2$
\begin{equation*}
 |\eta^n| = \tau^{-\alpha} |U_h^{n} - \overline{U}_h^n|\le\epsilon\tau^{-\alpha}t_n^{\frac\alpha2}|U_h^{n,0}-U_h^{n} |,
\end{equation*}
which together with the choice of $ U_h^{n,0}$ in \eqref{eqn:Uhn0} implies
\begin{equation}\label{eqn:eta-n}
  |\eta^n| \le \epsilon \tau^{1-\alpha} t_n^{\frac\alpha2}   \big(|\bar \partial_\tau   e^n|+|\bar \partial_\tau   e^{n-1}| + \tau|  \bar \partial_\tau^2u_h(t_n)|\big).
\end{equation}
By Lemma \ref{lem:est-uh-weight},
\begin{equation*}
 |\bar\partial_\tau^2 u_h(t_j)|\leq ct_j^{\alpha-2}\|Av\|_{L^2(\Omega)},\quad j>2.
\end{equation*}
This and Lemma \ref{lem:basic-est1} give
\begin{equation}\label{eqn:est-y1}
\begin{split}
 &\tau^{3-\alpha} \sum_{j=3}^n (t_{n+1}-t_j)^{\frac\alpha2-1} t_j^{\frac\alpha2}|  \bar \partial_\tau^{2}u_h(t_j)|\\
\le&c\tau^2 \sum_{j=3}^n (t_{n+1}-t_j)^{\frac\alpha2-1} t_j^{\frac\alpha2-1}  \|Av\|_{L^2(\Omega)}\\
\le&c \tau t_n^{\alpha-1}\|Av\|_{L^2(\Omega)}.
\end{split}
\end{equation}

\noindent{\bf Step 3}: Bound explicitly the term ${\rm II}$.
The estimates \eqref{eqn:eta-n} and \eqref{eqn:est-y1} imply
\begin{align*}
   \tau \sum_{j=1}^n(t_{n+1}-t_j)^{\frac{\alpha}{2}-1}|\eta^j|&\le  c\epsilon \tau t_n^{\alpha-1} \| Av\|_{L^2(\Omega)} \\
      &\quad+ c\epsilon \tau^{2-\alpha} \sum_{j=2}^n (t_{n+1}-t_j)^{\frac\alpha2-1} t_j^{\frac\alpha2}   |\bar \partial_\tau  e^j|.
\end{align*}
By the associativity identity $\bar \partial_\tau^{1-\alpha} \bar \partial_\tau^{\alpha} e^j=\bar \partial_\tau e^j$
and Lemma \ref{lem:inverse}, we get
\begin{align*}
  \tau^{2-\alpha} \sum_{j=2}^n (t_{n+1}-t_j)^{\frac\alpha2-1} t_j^{\frac\alpha2}  |\bar \partial_\tau  e^j|
  &\le   c \tau  \sum_{j=2}^n (t_{n+1}-t_j)^{\frac\alpha2-1}   |  \bar \partial_\tau^{\alpha} e^j|.
\end{align*}
Further, by Lemma \ref{lem:disc-stab-2} and \eqref{eqn:est-om}, there holds
\begin{align*}
  &\tau \sum_{j=1}^n (t_{n+1}-t_j)^{\frac\alpha2-1} |\bar \partial_\tau^\alpha e^j|\\
\le &c\ell_n \tau  \sum_{j=1}^n (t_{n+1}-t_j)^{\frac\alpha2-1} (\|(I+\tau^\alpha A_h)^{-\frac12}\omega^j\|_{L^2(\Omega)} + |\eta^j|)\\
\le &c\ell_n\tau  \sum_{j=1}^n (t_{n+1}-t_j)^{\frac\alpha2-1} |\eta^j| + c\ell_n^2\tau t_n^{\alpha-1}\|Av\|_{L^2(\Omega)}.
\end{align*}
The last three estimates together lead to
\begin{equation*}
\tau \sum_{j=1}^n(t_{n+1}-t_j)^{\frac{\alpha}{2}-1}|\eta^j| \leq
c\epsilon \tau t_n^{\alpha-1} \ell_n^2\| Av\|_{L^2(\Omega)} + c\epsilon\ell_n\tau \sum_{j=1}^n (t_{n+1}-t_j)^{\frac\alpha2-1} |\eta^j|,
\end{equation*}
which upon choosing a sufficiently small $\delta$ and {noting $\epsilon=\frac{\delta}{1-\delta}\ell_n^{-1}$} implies
\begin{equation*}
\tau \sum_{j=1}^n(t_{n+1}-t_j)^{\frac{\alpha}{2}-1}|\eta^j| \leq
c \tau t_n^{\alpha-1}\ell_n\| Av\|_{L^2(\Omega)}.
\end{equation*}
This, and the estimates \eqref{eqn:est-om}--\eqref{eqn:est-eta} complete the proof.\qed
\end{proof}

\subsection{Nonsmooth initial data}
Now we turn to nonsmooth initial data, i.e., $v\in L^2\II$. First we give a weighted estimate on the
time stepping scheme \eqref{eqn:fully}. The weight $t_n$ in the estimate is to compensate the strong
singularity of the summands.
\begin{lemma}\label{lem:stab2}
If $e^n\in X_h$ satisfies $e^0=0$ and
$\bar\partial_\tau^\alpha e^n + A_h e^n = \sigma^n$, $n=1,\ldots,N,$
then
\begin{equation*}
    t_n \|e^n\|_{L^2\II} \le c \tau \sum_{j=1}^n\big( \|A_h^{-1} \sigma^j\|_{L^2\II}+ (t_{n+1}-t_j)^{\frac\alpha2-1}  t_j \|  A_h^{-\frac12}\sigma^j \|_{L^2(\Omega)}\big).
\end{equation*}
\end{lemma}
\begin{proof}
Using  \eqref{eqn:sol-fully} and the splitting $t_n=(t_n-t_j)+t_j$, we have
\begin{equation*}
  t_ne^n 
  = \tau\sum_{j=1}^n(t_n-t_j)E_{h,\tau}^{n-j}\sigma^j + \tau\sum_{j=1}^nt_jE_{h,\tau}^{n-j}\sigma^j
\end{equation*}
Then from Lemma \ref{lem:stab-sol-op}, we deduce
\begin{align*}
  t_n\|e^n\|_{L^2(\Omega)} 
    & \leq \tau\sum_{j=1}^n(t_n-t_j)\|A_hE_{h,\tau}^{n-j}\|\|A_h^{-1}\sigma^j\|_{L^2(\Omega)} \\
     &\quad + \tau \sum_{j=1}^nt_j\|A_h^\frac12E_{h,\tau}^{n-j}\|\|A_h^{-\frac12}\sigma^j\|_{L^2(\Omega)}\\
    & \leq c\tau\sum_{j=1}^n(t_n-t_j)(t_{n+1}-t_j)^{-1}\|A_h^{-1}\sigma^j\|_{L^2(\Omega)} \\
     &\quad + c\tau\sum_{j=1}^nt_j(t_{n+1}-t_j)^{\frac\alpha2-1}\|A_h^{-\frac12}\sigma^j\|_{L^2(\Omega)},
\end{align*}
from which the desired assertion follows directly.\qed
\end{proof}

\begin{lemma}\label{lem:bdd-eta}
Let $e^n\in X_h$ satisfy $e^0=0$ and
$\bar\partial_\tau^\alpha e^n + A_h e^n = \sigma^n$, $n=1,\ldots,N.$
Then with $\ell_n=\ln(1+t_n/\tau)$, there holds
\begin{align*}
  &\tau^{2-\alpha} \sum_{j=1}^n\big(t_j+(t_{n+1}-t_j)^{\frac\alpha2-1}t_j^2\big) |\bar \partial_\tau e^j|\\
  \leq& c\ell_nt_n\tau\sum_{j=1}^n\big(1+t_j(t_{n+1}-t_j)^{\frac{\alpha}{2}-1}\big)\|(I+\tau^\alpha A_h)^{-\frac12}\sigma^j\|_{L^2(\Omega)}.
\end{align*}
\end{lemma}
\begin{proof}
Let $a^\ell=|\bar\partial_\tau^\alpha e^\ell|$, and $I_h=(I+\tau^\alpha A_h)^{-\frac12}$.
The proof of Lemma \ref{lem:disc-stab-2} gives
\begin{align}\label{eqn:bdd-aj}
  |a^j| \leq \|I_h\sigma^j\|_{L^2(\Omega)} + \|I_hA_he^j\|_{L^2(\Omega)} + \tau^\alpha\|I_hA_h\bar\partial_\tau^\alpha e^j\|_{L^2(\Omega)}.
\end{align}
By the solution representation $e^j=\sum_{\ell=1}^jE_{h,\tau}^{j-\ell}\sigma^\ell$ and Lemma \ref{lem:stab-sol-op},
\begin{equation}\label{eqn:bdd-IhAhsigma}
  \|I_hA_he^j\|_{L^2(\Omega)}\leq c\tau\sum_{\ell=1}^j(t_{j+1}-t_{\ell})^{-1}\|I_h\sigma^\ell\|_{L^2(\Omega)}.
\end{equation}
Now by the identity $\bar\partial_\tau = \bar\partial_\tau^{1-\alpha}\bar\partial_\tau^\alpha$, since $e^0=0$, we have
\begin{align*}
  {\rm LHS} \leq & \tau \sum_{j=1}^n t_j\sum_{\ell=1}^j|b_{j-\ell}^{(1-\alpha)}|a^\ell + \tau \sum_{j=1}^n (t_{n+1}-t_j)^{\frac\alpha2-1}t_j^2\sum_{\ell=1}^j|b_{j-\ell}^{(1-\alpha)}| a^\ell:= {\rm I}_1 + {\rm I}_2.
\end{align*}
It suffices to bound the two terms ${\rm I}_1$ and ${\rm I}_2$ separately. For the first term ${\rm I}_1$,
Lemmas \ref{lem:bdd-b} and \ref{lem:basic-est1} give
\begin{align*}
  {\rm I}_1 &\leq ct_n\tau  \sum_{\ell=1}^na^\ell \sum_{j=\ell}^n (j+1-\ell)^{\alpha-2}
           \leq c\tau t_n \sum_{j=1}^n a^j.
\end{align*}
Meanwhile, the following inverse inequality is direct from Lemmas \ref{lem:bdd-b} and \ref{lem:basic-est1}:
\begin{align*}
   \tau^{\alpha}\sum_{j=1}^n\|I_hA_h\bar\partial_\tau^\alpha e^j\|_{L^2(\Omega)}
   \leq & c\sum_{j=1}^n\sum_{\ell=1}^j(j+1-\ell)^{-\alpha-1}\|I_hA_he^\ell\|_{L^2(\Omega)}\\
   &\leq c\sum_{j=1}^n \|I_hA_he^j\|_{L^2(\Omega)}.
\end{align*}
The last two estimates, \eqref{eqn:bdd-aj}--\eqref{eqn:bdd-IhAhsigma} and Lemma \ref{lem:basic-est1} imply
\begin{align*}
  {\rm I}_1 &\leq ct_n\tau \sum_{j=1}^n (\|I_h\sigma^j\|_{L^2(\Omega)} + \|I_hA_he^j\|_{L^2(\Omega)})\\
            & \leq ct_n\tau \sum_{j=1}^n \|I_h\sigma^j\|_{L^2(\Omega)} + ct_n\tau^2\sum_{j=1}^n\sum_{\ell=1}^j(t_{j+1}-t_\ell)^{-1}\|I_h\sigma^\ell\|_{L^2(\Omega)}\\
            & \leq c\ell_n t_n\tau \sum_{j=1}^n \|I_h\sigma^j\|_{L^2(\Omega)}.
\end{align*}
Next, we bound the term ${\rm I}_2$. By \eqref{eqn:bdd-aj}, the inner sum of the term ${\rm I}_2$ can be bounded by
\begin{align*}
   \sum_{\ell=1}^j|b_{j-\ell}^{(1-\alpha)}| a^\ell
  &\leq  \sum_{\ell=1}^j|b_{j-\ell}^{(1-\alpha)}|\big(\|I_h\sigma^\ell\|_{L^2(\Omega)}\\
     &\quad+\|I_hA_he^\ell\|_{L^2(\Omega)} +\tau^\alpha\|I_hA_h\bar\partial_\tau^\alpha e^\ell\|_{L^2(\Omega)}\big).
\end{align*}
Lemma \ref{lem:bdd-b}, changing the summation order and Lemma \ref{lem:basic-est1} imply the following inverse inequality (upon relabeling):
\begin{align*}
  & \tau^\alpha\sum_{\ell=1}^j|b_{j-\ell}^{(1-\alpha)}|\|I_hA_h\bar\partial_\tau^\alpha e^\ell\|_{L^2(\Omega)}\\
 \leq & c \sum_{\ell=1}^j(j+1-\ell)^{\alpha-2}\sum_{i=1}^{\ell}(\ell+1-i)^{-\alpha-1}\|I_hA_he^i\|_{L^2(\Omega)}\\
 \leq & c\sum_{i=1}^j (j+1-i)^{-\gamma^*}\|I_hA_he^i\|_{L^2(\Omega)},
\end{align*}
with $\gamma^*=\min(2-\alpha,1+\alpha)$. The last two estimates and \eqref{eqn:bdd-IhAhsigma} yield
\begin{align*}
   \sum_{\ell=1}^j|b_{j-\ell}^{(1-\alpha)}| a^\ell&\leq  c\sum_{\ell=1}^j(j+1-\ell)^{-\gamma^*}\big(\|I_h\sigma^\ell\|_{L^2(\Omega)}+\|I_hA_he^\ell\|_{L^2(\Omega)}\big)\\
     &\leq c\sum_{\ell=1}^j(j+1-\ell)^{-\gamma^*}\|I_h\sigma^\ell\|_{L^2(\Omega)}\\
      &\quad +c\sum_{\ell=1}^j(j+1-\ell)^{-\gamma^*}\sum_{k=1}^\ell(\ell+1-i)^{-1}\|I_h\sigma^k\|_{L^2(\Omega)}.
\end{align*}
Now by changing the summation order and using Lemma \ref{lem:basic-est1}, we deduce
\begin{align*}
  \sum_{\ell=1}^j|b_{j-\ell}^{(1-\alpha)}| a^\ell & \leq c\sum_{\ell=1}^j(j+1-\ell)^{-1}\|I_h\sigma^\ell\|_{L^2(\Omega)}.
\end{align*}
Consequently, with the splitting $t_j \leq (t_{j}-t_l) + t_\ell$,
\begin{align*}
  {\rm I}_2 & \leq c\tau \sum_{j=1}^n (t_{n+1}-t_j)^{\frac\alpha2-1}t_j^2
  \sum_{\ell=1}^j(j+1-\ell)^{-1}\|I_h\sigma^\ell\|_{L^2(\Omega)} \\
  & \leq  c t_n\tau\sum_{\ell=1}^n\|I_h\sigma^\ell\|_{L^2(\Omega)}
  \sum_{j=\ell}^n (t_{n+1}-t_j)^{\frac\alpha2-1}(t_j-t_\ell)(j+1-\ell)^{-1}\\
   &\quad +  ct_n\tau\sum_{\ell=1}^n\|I_h\sigma^\ell\|_{L^2(\Omega)}t_\ell
  \sum_{j=\ell}^n(t_{n+1}-t_j)^{\frac\alpha2-1}(j+1-\ell)^{-1}\\
  & \leq  c t_n\tau\sum_{\ell=1}^n(t_{n+1}-t_\ell)^{\frac{\alpha}{2}}\|I_h\sigma^\ell\|_{L^2(\Omega)}\\
  &\quad  +  ct_n\ell_n\tau\sum_{\ell=1}^nt_\ell(t_{n+1}-t_\ell)^{\frac\alpha2-1}\|I_h\sigma^\ell\|_{L^2(\Omega)}.
\end{align*}
Now relabeling and collecting the terms yield the desired assertion.\qed

\end{proof}

Next, we give a weighted estimate due to the local truncation error $\omega^k$.
\begin{lemma}\label{lem:bdd-omega}
Let $e^n\in X_h$ satisfy $e^0=0$ and
\begin{equation*}
   \bar\partial_\tau^\alpha e^n + A_h e^n = \omega^n,\quad n=1,\ldots,N,
\end{equation*}
where $\omega^n$, $n=1,\ldots,N$, are defined in \eqref{eqn:omega-eta}. Then with $\ell_n=\ln(1+t_n/\tau)$, there holds
\begin{align*}
  & \tau^{2-\alpha} \sum_{j=1}^n t_j |\bar \partial_\tau e^j| + \tau^{2-\alpha} \sum_{j=1}^n (t_{n+1}-t_j)^{\frac\alpha2-1}t_j^2 |\bar \partial_\tau e^j| \leq c\ell_n^2\tau^{2-\alpha} \|v\|_{L^2(\Omega)}.
\end{align*}
\end{lemma}
\begin{proof}
By applying the operator $\bar\partial_\tau$ to both sides of the defining equation for $e^n$
and the associativity of CQ, we obtain
\begin{equation*}
  \bar \partial_\tau e^n = \tau \sum_{k=1}^n E_{h,\tau}^{n-k} \bar \partial_\tau \omega^k.
\end{equation*}
Let $w_{j,n}=t_j+(t_{n+1}-t_j)^{\frac\alpha2-1}t_j^2$ be the weight. We split
LHS into two parts:
\begin{equation*}
  {\rm LHS} = \tau^{2-\alpha}w_{1,n} |\bar \partial_\tau e^1| + {\rm LHS}',\quad \mbox{with }{\rm LHS}' = \tau^{2-\alpha}\sum_{j=2}^nw_{j,n}|\bar \partial_\tau e^j|.
\end{equation*}
Since $e^0=0$, by Lemmas \ref{lem:tal2} and \ref{lem:stab-sol-op},
\begin{align*}
   \tau^{2-\alpha}w_{1,n} |\bar \partial_\tau e^1| & \leq \tau^{2-\alpha}w_{1,n}\|\bar \partial_\tau e^1\|_{L^2(\Omega)}
  + \tau^{2}w_{1,n}\| A_h\bar \partial_\tau e^1\|_{L^2(\Omega)}\\
     & = \tau^{2-\alpha}w_{1,n}\|E_{h,\tau}^{0}\omega^1\|_{L^2(\Omega)}
  + \tau^{2}w_{1,n}\| A_h E_{h,\tau}^0\omega^1\|_{L^2(\Omega)}\\
    & \leq c\tau^{2-\alpha}\|v\|_{L^2(\Omega)}.
\end{align*}
Thus it suffices to bound the sum ${\rm LHS}'$. Similarly,
\begin{align*}
 {\rm LHS}' &\le \tau^{2-\alpha} \sum_{j=2}^nw_{j,n}\|\bar \partial_\tau e^j\|_{L^2(\Omega)}
  + \tau^{2} \sum_{j=2}^nw_{j,n}\| A_h\bar \partial_\tau e^j\|_{L^2(\Omega)}
  :={\rm I}+{\rm II}.
\end{align*}
For the term ${\rm I}$, we further split it into two terms (with $m_j=[{j/2}]$, where $[\cdot]$ denotes
taking the integral part of a real number):
\begin{align*}
  {\rm I} &\leq \tau^{3-\alpha} \sum_{j=2}^n w_{j,n} \|\sum_{k=1}^{m_j} E_{h,\tau}^{j-k}\bar \partial_\tau \omega^k\|_{L^2(\Omega)} \\
   &\quad + \tau^{3-\alpha} \sum_{j=2}^n   w_{j,n} \|\sum_{k=m_j+1}^j E_{h,\tau}^{j-k}\bar \partial_\tau \omega^k\|_{L^2(\Omega)}
  : = {\rm I}_1 + {\rm I}_2.
\end{align*}
Then by the summation by parts formula
\begin{equation}\label{eqn:disc-int-part}
\sum_{k=0}^jf_k(g_{k+1}-g_{k}) + \sum_{k=1}^{j}g_k(f_k-f_{k-1}) = f_jg_{j+1}-f_0g_0,
\end{equation}
since $\omega^0=0$, there holds
\begin{equation*}
  \sum_{k=1}^{m_j} E_{h,\tau}^{j-k}\bar \partial_\tau \omega^k = \sum_{k=1}^{m_j-1}\big(\bar\partial_\tau E_{h,\tau}^{j-k}\big)\omega^k + \tau^{-1}E_{h,\tau}^{j-m_j}\omega^{m_j}.
\end{equation*}
This, the triangle inequality, and Lemma \ref{lem:stab-sol-op}, we have
\begin{align*}
{\rm I}_1 =&\tau^{3-\alpha} \sum_{j=2}^n w_{j,n}  \|\sum_{k=1}^{m_j-1} \big(\bar \partial_\tau E_{h,\tau}^{j-k}\big)\omega^k + \tau^{-1} {E_{h,\tau}^{j-m_j}\omega^{m_j}}\|_{L^2(\Omega)}\\
\le & \tau^{3-\alpha} \sum_{j=2}^n w_{j,n}  \sum_{k=1}^{m_j-1} \|\big(\bar \partial_\tau E_{h,\tau}^{j-k}\big)\omega^k\|_{L^2(\Omega)}
 +  \tau^{2-\alpha} \sum_{j=2}^n w_{j,n}   \| {E_{h,\tau}^{j-m_j}\omega^{m_j}}\|_{L^2(\Omega)}\\
\le &c \tau^{3-\alpha} \sum_{j=2}^n w_{j,n}  \sum_{k=1}^{m_j-1} (t_{j+1}-t_k)^{-2} \| A_h^{-1} \omega^k\|_{L^2(\Omega)}\\
 & + c \tau^{2-\alpha} \sum_{j=2}^nw_{j,n}(t_{j+1}-t_{m_j})^{-1} \| A_h^{-1} \omega^{m_j}\|_{L^2(\Omega)}.
\end{align*}
By Lemma \ref{lem:tal2}, $\|A_h^{-1}\omega^j\|_{L^2(\Omega)}\leq c\tau t_j^{-1}\|v\|_{L^2(\Omega)}$, and upon
substitution, Lemma \ref{lem:basic-est1} implies
\begin{align*}
 {\rm I}_1 &\le c \tau^{4-\alpha} \sum_{j=2}^n t_j^{-2}w_{j,n}  \sum_{k=1}^{m_j-1}t_k^{-1}\| v\|_{L^2(\Omega)}
 + c \tau^{3-\alpha} \sum_{j=2}^n w_{j,n}t_j^{-2} \|v\|_{L^2(\Omega)}\\
  &\le c\tau^{2-\alpha} \ell_n^2 \| v\|_{L^2(\Omega)}.
\end{align*}
Similarly, by Lemma \ref{lem:stab-sol-op}, Corollary \ref{cor:tal3} and Lemma
\ref{lem:basic-est1}, we deduce
\begin{align*}
{\rm I}_2
 \le& c\tau^{3-\alpha} \sum_{j=2}^nw_{j,n}\sum_{k=m_j+1}^j \|E_{h,\tau}^{j-k} \bar \partial_\tau \omega^k\|_{L^2(\Omega)}\\
\le& c\tau^{3-\alpha} \sum_{j=2}^nw_{j,n}\sum_{k=m_j+1}^j (t_{j+1}-t_k)^{-1} \|A_h^{-1}\bar \partial_\tau \omega^k\|_{L^2(\Omega)}\\
\le &c\tau^{4-\alpha} \sum_{j=2}^nw_{j,n}\sum_{k=m_j+1}^j (t_{j+1}-t_k)^{-1} t_k^{-2}\| v\|_{L^2(\Omega)}\\
\le &c\tau^{2-\alpha}\ell_n^2  \| v \|_{L^2(\Omega)}.
\end{align*}
Thus, ${\rm I} \le c\tau^{2-\alpha} \ell_n^2 \| v \|_{L^2(\Omega)}$.
In the same manner, we further split ${\rm II}$ into two terms
\begin{align*}
 {\rm II} \leq & \tau^{3} \sum_{j=2}^nw_{j,n} \|\sum_{k=1}^{m_j} E_{h,\tau}^{j-k} A_h\bar \partial_\tau \omega^k\|_{L^2(\Omega)} \\
   &+ \tau^{3} \sum_{j=2}^nw_{j,n}\|\sum_{k=m_j+1}^j E_{h,\tau}^{j-k} A_h\bar \partial_\tau \omega^k\|_{L^2(\Omega)}:={\rm II}_1+{\rm II}_2.
\end{align*}
For the term ${\rm II}_1$, we apply summation by
parts formula \eqref{eqn:disc-int-part}, triangle inequality, Lemmas \ref{lem:stab-sol-op}, \ref{lem:tal2} and
\ref{lem:basic-est1} to obtain
\begin{align*}
{\rm II}_1 =&\tau^{3} \sum_{j=2}^nw_{j,n}  \|\sum_{k=1}^{m_j-1}\big(\bar\partial_\tau E_{h,\tau}^{j-k}\big) A_h \omega^k + \tau^{-1}
{E_{\tau}^{j-m_j} A_h\omega^{m_j}}\|_{L^2(\Omega)}\\
\le & \tau^{3} \sum_{j=2}^nw_{j,n}  \sum_{k=1}^{m_j-1} \|\big(\bar\partial_\tau E_{h,\tau}^{j-k}\big)A_h \omega^k\|_{L^2(\Omega)}
 +  \tau^{2} \sum_{j=2}^nw_{j,n}   \| {E_{h,\tau}^{j-m_j}A_h\omega^{m_j}}\|_{L^2(\Omega)}\\
\le &c \tau^{3} \sum_{j=2}^nw_{j,n}  \sum_{k=1}^{m_j-1} (t_{j+1}-t_k)^{-2} \|  \omega^k\|_{L^2(\Omega)}\\
 &+ c \tau^{2} \sum_{j=2}^nw_{j,n}(t_{j+1}-t_{m_j})^{-1}  \| \omega^{m_j}\|_{L^2(\Omega)}\\
 \le &c \tau^{4} \sum_{j=2}^n t_j^{-2}w_{j,n}  \sum_{k=1}^{m_j-1} t_k^{-1-\alpha}\| v\|_{L^2(\Omega)}
 + c \tau^{3} \sum_{j=2}^n{w_{j,n}t_j^{-2-\alpha}}\|v\|_{L^2(\Omega)}\\
 \le& c\tau^{2-\alpha} \ell_n^2 \| v\|_{L^2(\Omega)},
\end{align*}
and likewise by Lemma \ref{lem:stab-sol-op} and Corollary \ref{cor:tal3},
\begin{align*}
{\rm II}_2
 &\le \tau^{3} \sum_{j=2}^n w_{j,n}\sum_{k=m_j+1}^j \|E_{\tau}^{j-k}A_h \bar \partial_\tau \omega^k\|_{L^2(\Omega)}\\
&\le c\tau^{3}\sum_{j=2}^n w_{j,n}\sum_{k=m_j+1}^j (t_{j+1}-t_k)^{-1} \| \bar \partial_\tau \omega^k\|_{L^2(\Omega)}\\
&\le c\tau^{4} \sum_{j=2}^n w_{j,n}\sum_{k=m_j+1}^j (t_{j+1}-t_k)^{-1} t_k^{-2-\alpha} \| v\|_{L^2(\Omega)}\\
&\le c\tau^{4-\alpha} \sum_{j=2}^n w_{j,n}t_j^{-2}\sum_{k=m_j+1}^j (t_{j+1}-t_k)^{-1} \| v\|_{L^2(\Omega)}\\
&\le c\tau^{2-\alpha}\ell_n^2  \| v \|_{L^2(\Omega)}.
\end{align*}
Thus, ${\rm II} \le c\tau^{2-\alpha} \ell_n^2 \| v \|_{L^2(\Omega)}$,
and the desired assertion follows. \qed
\end{proof}

Now we can state the error estimate for \eqref{eqn:fully2} with $v\in L^2(\Omega)$.
\begin{theorem}\label{thm:err-nonsmooth-ini}
Let $v\in L^2(\Omega)$ and assumption \eqref{eqn:Bhtau} hold. Let $U_h^n \equiv U_h^{n,M_n}$
be the solution to \eqref{eqn:Uhn0}--\eqref{eqn:fully2} with $f=0$ and $v_h=P_hv$, and let
$U_h^n= \overline U_h^n$ for $n=1,2$. Then with $\ell_n=\ln(1+t_n/\tau)$,
there exists a $\delta>0$ such that
\begin{equation*}
\|  U_h^n - u_h(t_n)  \|_{L^2(\Omega)} \le c  \tau t_n^{ -1} {\color{blue} \ell_n} \|v\|_{L^2(\Omega)},\quad \text{if}~c_0\kappa^{M_n} \le \delta \min(t_n,1)\ell_n^{-1}.
\end{equation*}
\end{theorem}
\begin{proof}
The proof employs \eqref{eqn:err-smooth-data}--\eqref{eqn:barUhn},
and the overall strategy is similar to that for Theorem \ref{thm:err-smooth-ini}. However, due
to lower solution regularity for $v\in L^2(\Omega)$, the requisite weighted estimates are different.
Below we sketch the main steps.

\noindent{\bf Step 1}: Bound $\|e^n\|_{L^2(\Omega)}$ by $|\eta^j|$s.
By \eqref{eqn:err-smooth-data} and Lemma \ref{lem:stab2},
\begin{align*}
  t_n \|  e^n  \|_{L^2(\Omega)}   \le&   c \tau \sum_{j=1}^n  \| A_h^{-1} \sigma^j \|_{L^2(\Omega)}
  + c \tau \sum_{j=1}^n (t_{n+1}-t_j)^{\frac\alpha2-1} t_j \| A_h^{-\frac12}\sigma^j  \|_{L^2(\Omega)}\\
    \leq &   c \Big(\tau \sum_{j=1}^n  \| A_h^{-1} \omega^j \|_{L^2(\Omega)} + \tau \sum_{j=1}^n (t_{n+1}-t_j)^{\frac\alpha2-1} t_j
     \| A_h^{-\frac12}\omega^j  \|_{L^2(\Omega)}\Big)\\
    & +   c\Big(\tau \sum_{j=1}^n  \| A_h^{-1}(I+\tau^\alpha A_h)\eta^j \|_{L^2(\Omega)}\\
    &+ \tau \sum_{j=1}^n (t_{n+1}-t_j)^{\frac\alpha2-1} t_j \| A_h^{-\frac12}(I+\tau^\alpha A_h)\eta^j  \|_{L^2(\Omega)}\Big): = {\rm I}+{\rm II}.
\end{align*}
For the term ${\rm I}$, Lemmas \ref{lem:tal2}(ii) and \ref{lem:basic-est1} lead to
\begin{equation*}
{\rm I} \le c\tau\ell_n \|v\|_{L^2(\Omega)}.
\end{equation*}
The estimate \eqref{eqn:eta} allows simplifying the term ${\rm II}$ to
\begin{align}\label{eqn:est-II-nonsmooth}
{\rm II} \leq c\tau\sum_{j=1}^nw_{j,n} |\eta^j|\quad \mbox{with } w_{j,n}=1+(t_{n+1}-t_j)^{\frac\alpha2-1}t_j.
\end{align}
The rest of the proof is to bound ${\rm II}$ under assumption \eqref{eqn:Bhtau}.

\noindent{\bf Step 2}: Bound the summand $|\eta^n|$.
Under assumption \eqref{eqn:Bhtau} and triangle inequality, there holds
\begin{equation*}
 |U_h^{n} - \overline{U}_h^n| \le c_0 \kappa^{M_n} \big( |U_h^{n,0} - U_h^{n}| + |U_h^{n} - \overline{U}_h^n | \big).
\end{equation*}
Next we choose $M_n$ such that $c_0\kappa^{M_n} \le \delta \min(t_n,1)\ell_n^{-1}$,
and let $\epsilon = \frac{\delta}{1-\delta} \ell_n^{-1}$. Then we have
\begin{equation*}
|U_h^{n} - \overline{U}_h^n| 
\le \epsilon t_n |U_h^{n,0} - U_h^{n} |,
\end{equation*}
and hence
\begin{equation*}
 |\eta^n| = \tau^{-\alpha} |U_h^{n} - \overline{U}_h^n|
 \le  \epsilon\tau^{-\alpha} t_n  |U_h^{n,0} - U_h^{n} |.
\end{equation*}
By the choice of $ U_h^{n,0}$ in \eqref{eqn:Uhn0}, $\eta^1=\eta^2=0$ and, for $n\geq 3$,
\begin{equation}\label{eqn:eta-nonsmooth}
  |\eta^n| \le c\epsilon \tau^{1-\alpha} t_n\big(|\bar \partial_\tau  e^n|+|\bar \partial_\tau  e^{n-1}| + \tau| \bar \partial_\tau^2u_h(t_n)|\big).
\end{equation}
Meanwhile, by Lemmas \ref{lem:est-uh-weight} and \ref{lem:basic-est1}, we have
\begin{equation}\label{eqn:est-uh}
    \tau^{3-\alpha} \sum_{j=3}^n t_jw_{j,n}|  \bar \partial_\tau^2 u_h(t_j)|
\le  c   \tau^{2-\alpha} \ell_n  \|v\|_{L^2(\Omega)}.
\end{equation}

\noindent{\bf Step 3}: Bound the term ${\rm II}$ explicitly.
It follows from \eqref{eqn:est-II-nonsmooth}--\eqref{eqn:est-uh} that
\begin{align*}
   {\rm II} \le  c\epsilon \tau^{2-\alpha} \ell_n \|v\|_{L^2(\Omega)}
    + c\epsilon \tau^{2-\alpha} \sum_{j=1}^n t_jw_{j,n}|\bar \partial_\tau e^j|.
\end{align*}
It follows from Lemmas \ref{lem:bdd-eta} and \ref{lem:bdd-omega}, invoked respectively for $\eta^j$ and $\omega^j$, that
\begin{align*}
   & \tau^{2-\alpha} \sum_{j=1}^n t_jw_{j,n}|\bar \partial_\tau e^j|
   \leq  c\ell_nt_n\tau\sum_{j=1}^nw_{j,n}|\eta^j| + c\tau^{2-\alpha}\ell_n^2\|v\|_{L^2(\Omega)}.
\end{align*}
The rest of the proof is identical with Theorem \ref{thm:err-smooth-ini}, and hence omitted. \qed
\end{proof}

\begin{remark}
The numerical solution $U_h^n$ by the time stepping scheme \eqref{eqn:fully}
satisfies \cite[Theorem 3.5]{JinLazarovZhou:SISC2016}
\begin{equation*}
  \|U_h^n-u_h(t_n)\| \leq \left\{\begin{array}{ll}
      c\tau t_n^{\alpha-1}\|Av\|_{L^2(\Omega)}, &\quad \mbox{if } v_h=R_hv,\\
      c\tau t_n^{-1} \|v\|_{L^2(\Omega)}, &\quad \mbox{if } v_h=P_hv.
      \end{array}\right.
\end{equation*}
The error estimates in Theorems \ref{thm:err-smooth-ini} and \ref{thm:err-nonsmooth-ini} for
\eqref{eqn:fully2} are comparable, up to a log factor $\ell_n$. However, the IIS \eqref{eqn:fully2}
does not require the exact solution of the resulting linear systems and thus can be more efficient.
\end{remark}

\section{Numerical experiments and discussions}\label{sec:numer}
Now we present numerical results to illustrate the theoretical results. The numerical experiments are performed
on the square $\Omega=(-1,1)^2$. In the computation, we first divide the interval $(-1,1)$ into $K$ equally
spaced subintervals of length $h=2/K$ so that the domain $\Omega=(-1,1)^2$ is divided into $K^2$ small squares,
and then obtain a uniform triangulation by connecting the diagonal of each small square. We divide the time
interval $[0,T]$ into a uniform grid with a time step size $\tau=T/N$. Since the semidsicrete solution $u_h$ is
not available in closed form, we compute a reference solution $u_h(t_n)$ by the corrected CQ generated by BDF3
\cite{JinLiZhou:2017} in time with $N=1000$ and $K=256$ in space. We compute the temporal error at $t_N=T$ by
\begin{equation*}
 e^N = \frac{\| U_h^N - u_h(t_N) \|_{L^2(\Omega)}}{\|   u_h(t_N) \|_{L^2(\Omega)}}.
\end{equation*}
In the IIS \eqref{eqn:fully2}, any iterative solver satisfying the contraction property \eqref{eqn:Bhtau} can
be employed. In this work, we employ the V-cycle multigrid method
with standard Jacobi or Gauss-Seidel smoothers to inexactly solve the linear systems, which is known to
satisfy \eqref{eqn:Bhtau} \cite[Theorem 11.4, p. 199]{Thomee:2006}. Multigrid type methods have been employed
in \cite{LinLuNgSun:2016,GasparRodrigo:2017}, but without error analysis for either smooth or
nonsmooth solutions. In the experiments, the spatial mesh size $h$ is fixed with $K=256$ so that the numerical
results focus on the temporal error.

\subsection{Example 1: smooth solutions}
First we consider problem \eqref{eqn:fde} with $A=-5\Delta$, $T=1$, $v=0$ and $f(x,t) = t^2(1+x_1)(1-x_1)(1+x_2)(1-x_2).$
The source term $f$ satisfies compatibility conditions: $f(0)=f'(0)=0$ and {$f\in C^2([0,T], D(A))$}.
Thus the solution $u$ satisfies the regularity assumption in Theorem \ref{thm:iter-smooth} (see Remark \ref{rmk:smooth-sol}),
and accordingly, the number $M_n$ of iterations may be taken to be uniform in time, which is sufficient to preserve the
desired first-order convergence.

\begin{table}[hbt!]
\centering
\caption{$L^2$ errors $e^N$ for Example 1 with $K=128$, point Jacobi smoother.}\label{tab:exp1}
{\setlength{\tabcolsep}{7pt}
	\begin{tabular}{|c|c|cccccc|}
		\hline
		$\alpha$  &$M_n\backslash N$ &$10$ &$20$ & $40$ & $80$ & $160$ & $320$   \\
		\hline
		     &   $1$      & 2.73e-3 & 5.46e-4 & 9.26e-5 & 4.41e-5 & 3.43e-5   & 2.16e-5\\
   & & &  2.32 & 2.56 &1.07 &0.36 &0.66\\
   \cline{2-8}
		 &   $2$      & 3.11e-4 & 2.37e-4 & 1.47e-4 & 8.83e-5 & 5.00e-5   & 2.55e-5\\
   & & &  0.39 & 0.69 &0.73 &0.82 &0.97\\
   \cline{2-8}
0.2 		     &   $3$     & 6.67e-4 & 3.35e-4 & 1.79e-4 & 9.61e-5 & 5.07e-5   & 2.57e-5\\
   & & &  0.99 & 0.91 &0.89 &0.92 &0.98\\
    \cline{2-8}
   		     &   $\infty$     & 8.31e-4 & 4.18e-4 & 2.09e-4 & 1.04e-4 & 5.24e-5   & 2.62e-5\\
   & & &  0.99 & 1.00 &1.00 &1.00 &1.00\\
    \hline
		     &   $1$      & 1.43e-3 & 2.06e-4 & 2.93e-4 & 2.10e-4 & 1.24e-4   & 6.64e-5\\
   & & &  2.79 & -0.51 &0.48 &0.76 &0.90\\
     \cline{2-8}
	 	 &   $2$      & 1.70e-3 & 9.41e-4 & 4.99e-4 & 2.66e-4 & 1.39e-4   & 6.98e-5\\
  & & &  0.85 & 0.92 &0.91 &0.94 &0.99\\
   \cline{2-8}
 0.5 		     &   $3$      & 2.07e-3 & 1.04e-3 & 5.31e-4 & 2.74e-4 & 1.39e-4  & 7.02e-5\\
   & & &  0.99 & 0.97 &0.95 &0.98 &0.99\\
      \cline{2-8}
		     &   $\infty$      & 2.23e-3 & 1.12e-3 & 5.63e-4 & 2.82e-4 & 1.41e-4  & 7.05e-5\\
   & & &  0.99 & 1.00 &1.00 &1.00 &1.00\\
	\hline
		     &   $1$      & 4.10e-4 & 8.79e-4 & 6.52e-4 & 3.94e-4 & 2.17e-4    & 1.12e-4\\
   & & &  -1.10 & 0.43 &0.73 &0.86 &0.96\\
     \cline{2-8}
		  &   $2$      & 3.13e-3 & 1.66e-3 & 8.58e-4 & 4.47e-4 & 2.26e-4     & 1.14e-4 \\
   & & &  0.92 & 0.95 &0.94 &0.98 &0.99\\
    \cline{2-8}
0.8		     &   $3$      & 3.49e-3 & 1.75e-3 & 8.85e-4 & 4.53e-4 & 2.28e-4     & 1.14e-4\\
   & & &  1.00 & 0.98 &0.97 &0.99 &1.00\\
    \cline{2-8}
   &   $\infty$      & 3.64e-3 & 1.83e-3 & 9.14e-4 & 4.58e-4 & 2.29e-4     & 1.14e-4\\
   & & &  0.99 & 1.00 &1.00 &1.00 &1.00\\
	\hline
   \end{tabular}}
\end{table}

\begin{table}[hbt!]
\centering
\caption{$L^2$ errors $e^N$ for Example 1 with $K=128$, Gauss-Seidel smoother.}\label{tab:exp1-GS}
{\setlength{\tabcolsep}{7pt}
	\begin{tabular}{|c|c|cccccc|}
		\hline
		$\alpha$  &$M_n\backslash N$ &$10$ &$20$ & $40$ & $80$ & $160$ & $320$   \\
		\hline
		     &   $1$      & 2.61e-3 & 4.46e-4 & 2.40e-5 & 5.20e-5 & 3.92e-5   & 2.29e-5\\
   & & &  2.55 & 4.22 &-1.12 &0.41 &0.77\\
   \cline{2-8}
		 &   $2$      & 3.97e-4 & 3.07e-4 & 1.82e-4 & 9.80e-5 & 5.08e-5   & 2.58e-5\\
   & & &  0.37 & 0.76 &0.89 &0.95 &0.98\\
   \cline{2-8}
0.2 		     &   $3$     & 7.58e-4 & 4.00e-4 & 2.05e-4 & 1.04e-4 & 5.21e-5   & 2.61e-5\\
   & & &  0.93 & 0.96 &0.98 &0.99 &1.00\\
        \cline{2-8}
   		     &   $\infty$     & 8.31e-4 & 4.18e-4 & 2.09e-4 & 1.04e-4 & 5.24e-5   & 2.62e-5\\
   & & &  0.99 & 1.00 &1.00 &1.00 &1.00\\
    \hline
    		     &   $1$      & 1.31e-3 & 2.75e-4 & 3.50e-4 & 2.30e-4 & 1.28e-4   & 6.75e-5\\
   & & &  2.26 & -0.35 &0.61 &0.84 &0.93\\
     \cline{2-8}
	 	 &   $2$      & 1.80e-3 & 1.02e-3 & 5.39e-4 & 2.76e-4 & 1.40e-4   & 7.03e-5\\
   & & &  0.82 & 0.92 &0.96 &0.98 &0.99\\
   \cline{2-8}
 0.5		     &   $3$      & 2.16e-3 & 1.11e-3 & 5.59e-4 & 2.81e-4 & 1.41e-4  & 7.05e-5\\
   & & &  0.99 & 0.97 &0.95 &0.98 &0.99\\
      \cline{2-8}
		     &   $\infty$      & 2.23e-3 & 1.12e-3 & 5.63e-4 & 2.82e-4 & 1.41e-4  & 7.05e-5\\
   & & &  0.99 & 1.00 &1.00 &1.00 &1.00\\
	\hline
		     &   $1$      & 4.04e-4 & 9.71e-4 & 7.12e-4 & 4.12e-4 & 2.19e-4    & 1.13e-4\\
   & & &  -1.27 & 0.45 &0.79 &0.91 &0.96\\
     \cline{2-8}
		  &   $2$      & 3.23e-3 & 1.74e-3 & 8.98e-4 & 4.55e-4 & 2.28e-4     & 1.14e-4 \\
   & & &  0.89 & 0.95 &0.98 &0.99 &1.00\\
    \cline{2-8}
0.8		     &   $3$      & 3.58e-3 & 1.81e-3 & 9.12e-4 & 4.57e-4 & 2.29e-4     & 1.14e-4\\
   & & &  0.98 & 0.99 &1.00 &1.00 &1.00\\
    \cline{2-8}
   &   $\infty$      & 3.64e-3 & 1.83e-3 & 9.14e-4 & 4.58e-4 & 2.29e-4     & 1.14e-4\\
   & & &  0.99 & 1.00 &1.00 &1.00 &1.00\\
	\hline
   \end{tabular}}
\end{table}

We present numerical results  for different values of the fractional order $\alpha$ and the number $M_n$ per iteration in
Tables \ref{tab:exp1} and \ref{tab:exp1-GS} obtained by the IIS \eqref{eqn:fully2} with point Jacobi and Gauss-Seidel
smoothers, respectively, where the choice $M_n=\infty$ corresponds to the direct solver for the linear
system at each time level. In each small block of the tables, the numbers under the errors denote the log
(with a base 2) of the ratio between the errors at consecutive time step sizes, and the theoretical
value is one for a first-order convergence. We observe that for all three $\alpha$ values, a steady
convergence for $M_n=2$ and $M_n=3$, however, the results for $M_n=1$ suffer from severe numerical
instability, as indicated by wild oscillations and large deviation from one. This observation holds for both Jacobi and
Gauss-Seidel smoothers, and agrees well with Theorem \ref{thm:iter-smooth}, which predicts that a
steady convergence of the scheme \eqref{eqn:fully2} requires a fixed but sufficiently large number of
iterations at all time levels for smooth solutions. Naturally, when the number $M_n$ is
sufficiently large, the obtained numerical solutions converge to that by the direct solver, which is
clearly observed in Tables \ref{tab:exp1} and \ref{tab:exp1-GS}. Surprisingly, although the convergence
of the incomplete iterative scheme becomes more steady as the number $M_n$ of iterations per time step
increases, the error does not decrease monotonically. That is, the incomplete iteration may actually improve the
accuracy of the numerical solution. The precise mechanism of the surprising phenomenon remains unclear.

\subsection{Example 2: nonsmooth solutions}
Next we consider problem \eqref{eqn:fde} with $A=-5\Delta$, $T=1$, $f=0$ and
$$v(x,y)=\chi_{(-1,0)}(x)+\chi_{(-1,0)}(y).$$
The initial data $v$ is piecewise constant and hence $v\in H^{\frac12-\epsilon}(\Omega)$ for any small $\epsilon>0$.
The number $M_n$ of iterations in the scheme \eqref{eqn:fully2} is taken to be (with integers $a,b\ge0$)
\begin{equation*}
 M_n = a + b \log_2(t_n^{-1}), \quad n > 2.
\end{equation*}

\begin{table}[hbt!]
\centering
\caption{$L^2$ errors $e^N$ for Example 2 with $a=3$ and $K=128$, point Jacobi smoother.}\label{tab:exp2}
{\setlength{\tabcolsep}{7pt}
	\begin{tabular}{|c|c|cccccc|}
		\hline
		$\alpha$  &$b\backslash N$ &$10$ &$20$ & $40$ & $80$ & $160$ & $320$   \\
		\hline
		     &   $0$      & 1.12e-2 & 5.61e-3 & 2.90e-3 & 1.47e-3 & 6.64e-4   & 3.25e-4\\
   & & &  1.06 & 0.95 &0.98 &1.15 &1.03\\
   \cline{2-8}
		 &   $3$      & 1.17e-2 & 5.71e-3 & 2.90e-3 & 1.57e-3 & 8.15e-4   & 3.67e-4\\
 0.2   & & &  1.03 & 0.98 &0.89 &0.94 &1.14\\
   \cline{2-8}
		     &   $6$     & 1.17e-2 & 5.77e-3 & 2.88e-3 & 1.42e-3 & 6.89e-4   & 3.49e-4\\
   & & &  1.02 & 1.00 &1.02 &1.05 &0.98\\
      \cline{2-8}
		     &   $\infty$     & 1.19e-2 & 5.85e-3 & 2.90e-3 & 1.45e-3 & 7.22e-4   & 3.61e-4\\
   & & &  1.02 & 1.01 &1.01 &1.00 &1.00\\
    \hline
		     &   $0$      & 3.80e-2 & 1.74e-2 & 9.74e-3 & 5.39e-3 & 2.55e-3   & 1.95e-3\\
   & & &  1.12 & 0.84 &0.85 &1.08 &0.39\\
     \cline{2-8}
	 	 &   $3$      & 3.82e-2 & 1.82e-2 & 9.52e-3 & 5.50e-3 & 2.80e-3   & 1.15e-3\\
 0.5   & & &  1.07 & 0.94 &0.80 &0.98 &1.27\\
   \cline{2-8}
		     &   $6$      & 3.84e-2 & 1.87e-2 & 9.36e-3 & 4.40e-3 & 2.20e-3  & 1.17e-3\\
   & & &  1.04 & 1.00 &1.09 &1.00 &0.90\\
      \cline{2-8}
		     &   $\infty$      & 3.94e-2 & 1.92e-2 & 9.47e-3 & 4.70e-3 & 2.34e-3  & 1.17e-3\\
   & & &  1.04 & 1.02 &1.01 &1.00 &1.00\\
	\hline
		     &   $0$      & 7.87e-2 & 3.29e-2 & 2.43e-2 & 1.16e-2 & 4.52e-3    & 4.03e-3\\
   & & &  1.26 & 0.44 &1.07 &1.36 &0.17\\
     \cline{2-8}
		  &   $3$      & 8.00e-2 & 3.70e-2 & 2.14e-2 & 1.12e-2 & 4.94e-3     & 2.54e-3 \\
  0.8 & & &  1.11 & 0.79 &0.94 &1.18 &0.96\\
    \cline{2-8}
		     &   $6$      & 8.12e-2 & 3.96e-2 & 2.01e-2 & 9.46e-3 & 5.40e-3     & 2.47e-3\\
   & & &  1.04 & 0.97 &1.09 &0.81 &1.13\\
       \cline{2-8}
		     &   $\infty$      & 8.75e-2 & 4.15e-2 & 2.03e-2 & 1.00e-2 & 4.97e-3     & 2.48e-3\\
   & & &  1.07 & 1.04 &1.02 &1.01 &1.00\\
	\hline
   \end{tabular}}
\end{table}

\begin{table}[hbt!]
\centering
\caption{$L^2$ errors $e^N$ for Example 2 with $K=128$, $a=1$ and $b=0$, Gauss-Seidel smoother.}\label{tab:exp3}
{\setlength{\tabcolsep}{7pt}
	\begin{tabular}{|c|cccccc|}
		\hline
		$\alpha\backslash N$  &$10$ &$20$ & $40$ & $80$ & $160$ & $320$   \\
		\hline
	 0.2 	    & 1.12e-2 & 5.71e-3 & 2.86e-3 & 1.43e-3 & 7.17e-4   & 3.58e-4\\
               & &  0.97 & 0.99 &1.00 &1.00 &1.00\\
	\hline
	 0.5 	     & 3.63e-2 & 1.84e-2 & 9.20e-3 & 4.58e-3 & 2.29e-3  & 1.14e-3\\
               & &  0.98 & 1.00 &1.00 &1.00 &1.00\\
	\hline
	 0.8 	    & 7.70e-2 & 3.89e-2 & 1.94e-2 & 9.71e-3 & 4.90e-3     & 2.47e-3\\
              & &  0.98 & 1.01 &1.00 &0.99 &0.99\\
	\hline
   \end{tabular}}
\end{table}

The numerical results for the example obtained with the scheme \eqref{eqn:fully2} with the Jacobi and Gauss-Seidel
smoothers are presented in Tables \ref{tab:exp2} and \ref{tab:exp3}, respectively. With the Jacobi smoother, it
is observed that with a fixed number of iterations at each time level (e.g., $M_n=3$), the IIS
\eqref{eqn:fully2} can fail to maintain the first order convergence, especially for $\alpha$ values close to one.
In contrast, surprisingly, for $\alpha$ value close to zero, even a fixed number of iterations tend to suffice
the desired first-order convergence, despite the low regularity of the solution. It might be related to
the fact that for small fractional order $\alpha$, the solution $u$ reaches a ``quasi''-steady state (before the
asymptotic regime) very rapidly, and thus the solution at neighboring time steps essentially reduces to very
similar elliptic problems. However, the precise mechanism of the interesting observation remains elusive. By
increasing the number $M_n$ of iterations slightly for small $t_n$, one can restore the desired $O(\tau)$
convergence rate of backward Euler CQ, which agree well with Theorems \ref{thm:err-smooth-ini} and \ref{thm:err-nonsmooth-ini}. By
changing Jacobi smoother to Gauss-Seidel smoother, the performance of the IIS \eqref{eqn:fully2} is
significantly enhanced, since one iteration at each time level is sufficient to maintain the desired accuracy.
The numerical results for Examples 1 and 2 show very clearly the potentials of the scheme \eqref{eqn:fully2} in
speeding up the numerical solution of the subdiffusion model with both smooth and nonsmooth solutions.

\section{Conclusions}
In this work, we have developed an efficient incomplete iterative scheme for the subdiffusion model. It
employs an iterative solver to solve the linear systems inexactly, and is straightforward to implement.
Further, we provided theoretical analysis of the scheme under a standard contraction assumption on the
iterative solver (in a weighted norm), and proved that it can indeed maintain the accuracy of the time
stepping scheme, provided the number of iterations at each time level is properly chosen, on which the
analysis has provided useful guidelines. The numerical experiments with standard multigrid methods
fully support the theoretical analysis and indicate that it can indeed significantly reduce the computational
cost of the time-stepping scheme.

In the context of nonsmooth data, the analysis of the incomplete iterative scheme \eqref{eqn:fully2} only
covers backward Euler convolution quadrature for the homogeneous problem. It is of much interest to extend the
analysis to other practically important scenarios, e.g., inhomogeneous problems and nonlinear problems, and high-order
time-stepping schemes, e.g., corrected L1 scheme and convolution
quadratures generated by BDF$k$ ($k\geq2$) and Runge-Kutta methods. In addition, the computational complexity and
memory requirement of the scheme can be further reduced by adopting suitable fast approximations to the convolution
\cite{AlpertGreengardHagstrom:2000,LubichSchadle:2002,JiangZhang:2017}.

\appendix

\section{Basic estimates}

\begin{lemma}\label{lem:basic-est1}
For $\beta,\gamma\geq 0$, there holds
\begin{equation*}
  \sum_{i=1}^n(n+1-i)^{-\beta}i^{-\gamma} \leq \left\{\begin{array}{ll}
      cn^{\max(1-\gamma,0)-\beta}, & 0\leq \beta<1,\gamma \neq1,\\
      cn^{-\beta}\ln (1+n), & 0\leq \beta\leq1,\gamma=1,\\
      c{n^{-\min(\beta,\gamma)}},       & \beta>1, \gamma>1.
  \end{array}\right.
\end{equation*}
\end{lemma}
\vspace{-0.3cm}
\begin{proof}
We denote by $[\cdot]$ the integral part of a real number. Then
\begin{equation*}
  \sum_{i=1}^n(n+1-i)^{-\beta}i^{-\gamma} =   \sum_{i=1}^{[{n/2}]}(n+1-i)^{-\beta}i^{-\gamma} +   \sum_{i=[{n/2}]+1}^n(n+1-i)^{-\beta}i^{-\gamma}:={\rm I}+{\rm II}.
\end{equation*}
Then, by the trivial inequalities: for $1\leq i\leq [{n/2}] $, there holds $(n+1-i)^{-\beta}\leq cn^{-\beta}$ and for $[{n/2}]+1\leq i\leq n$, there holds $i^{-\gamma}\leq cn^{-\gamma}$, we deduce
\begin{align*}
    {\rm I} \leq  cn^{-\beta} \sum_{i=1}^{[{n/2}]}i^{-\gamma}\quad \mbox{and}\quad
    {\rm II}\leq cn^{-\gamma} \sum_{i=[{n/2}]+1}^n(n+1-i)^{-\beta}.
\end{align*}
Simple computation gives $\sum_{i=1}^ji^{-\gamma}\leq cj^{\max(1-\gamma,0)}$ if $\gamma\neq1$ and
$\sum_{i=1}^ji^{-1}\leq c\ln(j+1)$.
Combining these estimates yields the desired assertion.\qed
\end{proof}

Next we give an upper bound on the CQ weights $b_j^{(\alpha)}$.
\begin{lemma}\label{lem:bdd-b}
For the weights $b_j^{(\alpha)}$, $|b_j^{(\alpha)}| \le e^{2\alpha}(j+1)^{-\alpha-1}$.
\end{lemma}
\begin{proof}
The weight $b_j^{(\alpha)}$ is given by $b_0^{(\alpha)}=1$ and $b_j^{(\alpha)}= -\Pi_{\ell=1}^j(1-\frac{1+\alpha}{\ell})$ for any $j\geq1$.
Note the elementary inequality $\ln (1-x)\leq -x$ for any $x\in(0,1)$, and the estimate
$\sum_{\ell=1}^j\ell^{-1}\geq \int_1^{j+1}s^{-1}\d s = \ln (j+1).$ Since $\ln \alpha = \ln (1-(1-\alpha))\leq \alpha-1$, for any $j\geq1$,
\begin{align*}
 \ln |b_j^{(\alpha)}| & = \ln\alpha+\sum_{\ell=2}^j \ln \left(1-\frac{1+\alpha}{\ell}\right) \le \ln \alpha - \sum_{\ell=2}^j\frac{1+\alpha}{\ell} \\
  & = \ln\alpha + (1+\alpha)  - \sum_{\ell=1}^j\frac{1+\alpha}{\ell} \le 2\alpha-(1+\alpha) \ln (j+1).
 \end{align*}
This completes the proof of the lemma. \qed
\end{proof}

\section{Proof of Lemmas \ref{lem:est-uh-weight} and \ref{lem:tal2}}\label{app:reg}

In this part, we provide the proof of Lemmas \ref{lem:est-uh-weight} and \ref{lem:tal2}. The proof of
Corollary \ref{cor:tal3} is identical with that for Lemma \ref{lem:tal2} and thus it is omitted.
The proof relies on the discrete Laplace transform, and the following two well-known estimates
\begin{align}
  \quad c_1 |z| &\le |\delta_\tau(e^{-z\tau})| \le c_2|z| \quad \forall z\in \Gamma_{\theta,\delta}^\tau,\label{eqn:gen}\\
  |\delta_\tau(e^{-z\tau})| &\le |z| \sum_{k=1}^\infty \frac{|z\tau|^{k-1}}{k!} \leq |z|e^{|z|\tau}, \quad \forall z\in \Sigma_\theta,\label{eqn:est-kernel}
\end{align}
and the resolvent estimate: for any $\theta\in(\pi/2,\pi)$,
\begin{equation}\label{eqn:resol}
  \|(z+A_h)^{-1}\|\leq c|z|^{-1},\quad \forall z\in \Sigma_\theta.
\end{equation}

Now we can give the proof of Lemma \ref{lem:est-uh-weight}.
\begin{proof}[of Lemma \ref{lem:est-uh-weight}]
By Laplace transform, $w_h(t_n)=\bar\partial_\tau^2 u_h(t_n)$ is given by
\begin{align*}
  w_h(t_n) = \frac{1}{2\pi i} \int_{\Gamma_{\theta,\delta}}\delta_\tau(e^{-z\tau})^2e^{zt_n}K(z)v_h \d z,\quad \mbox{with }K(z) = z^{\alpha-1}(z^\alpha+A_h)^{-1}.
\end{align*}
We split the contour $\Gamma_{\theta,\delta}$ into $\Gamma_{\theta,\delta}^\tau$ and $\Gamma_{\theta,\delta}
\setminus\Gamma_{\theta,\delta}^\tau$, and denote the corresponding integral by ${\rm I}$ and ${\rm II}$,
respectively. We discuss the cases $v\in L^2(\Omega)$ and $v\in D(A)$, separately.

\noindent Case (i): $v\in L^2(\Omega)$. By \eqref{eqn:gen} and \eqref{eqn:resol},
$\|K(z)\| \leq c$ for $z\in \Gamma_{\theta,\delta}^\tau$. Then choosing
$\delta=c/t_n$ in $\Gamma_{\theta,\delta}^\tau$ gives
\begin{equation*}
\| {\rm I}\|_{L^2(\Omega)}
   \le c \|  v_h \|_{L^2(\Omega)} \Big(\int_{\frac{c}{t_n}}^{\frac{\pi\sin\theta}{\tau}} \rho e^{t_n\rho\cos\theta} \,\d\rho
   + \int_{-\theta}^\theta t_n^{-2} \,\d\varphi \Big)
     \le c t_n^{-2}  \|  v_h \|_{L^2(\Omega)}.
\end{equation*}
For any $z=\rho e^{\pm\mathrm{i}\theta}\in \Gamma_{\theta,\delta}\setminus\Gamma_{\theta,\delta}^\tau$,
by the estimates \eqref{eqn:est-kernel} and \eqref{eqn:resol}, $  \|K(z)\| \leq ce^{2\rho\tau}.$
By choosing $\theta\in(\pi/2,\pi)$ sufficiently close to $\pi$, we deduce
\begin{align*}
\|{\rm II}\|_{L^2(\Omega)}
  \le  c \|  v_h\|_{L^2(\Omega)}  \int_{\frac{\pi\sin\theta}{\tau}}^\infty e^{\rho(\cos\theta t_n+2\tau)} \rho \,\d\rho \le ct_n^{-2} \|  v_h\|_{L^2(\Omega)}.
\end{align*}
Thus, $\|\bar\partial_\tau^2u_h(t_n)\|\le c t_n^{-2}\|v_h\|_{L^2(\Omega)}.$
Next, by the identity $A_h(z^\alpha+A_h)^{-1}=I-z^\alpha(z^\alpha +A_h)$ and \eqref{eqn:resol},
$\|A_hK(z)\|\leq |z|^{\alpha-1}$ for $z\in \Sigma_\theta$. Then repeating the argument gives
\begin{equation*}
  \tau^\alpha \|A_h \bar\partial_\tau^2 u_h(t_n)\| \le c \tau^\alpha t_n^{-2-\alpha} \| v_h\|_{L^2(\Omega)} \le ct_n^{-2} \| v_h  \|_{L^2(\Omega)} .
\end{equation*}
Then the assertion for the case $v\in L^2(\Omega)$ follows from the triangle inequality.

\noindent Case (ii): $v\in D(A)$. Simple computation gives the identity $K(z)v_h=z^{\alpha-1}(z^\alpha +A_h)^{-1}v_h
= z^{-1}v_h - z^{-\alpha}(z^\alpha+A_h)^{-1}A_hv_h$. Thus, we have
\begin{equation*}
  w_h(t_n)   = - \frac{1}{2\pi i} \int_{\Gamma_{\theta,\delta}}e^{zt_n} \delta_\tau(e^{-z\tau})^2 z^{-\alpha}K(z)A_h v_h \d z,
\end{equation*}
in which we split the contour $\Gamma_{\theta,\delta}$ into $\Gamma_{\theta,\delta}^\tau$ and $\Gamma_{\theta,\delta}\setminus
\Gamma_{\theta,\delta}^\tau$, and accordingly the integral. Then the rest of the proof follows from the estimates \eqref{eqn:gen},
\eqref{eqn:est-kernel} and \eqref{eqn:resol} as before.\qed
\end{proof}

Last, we prove Lemma \ref{lem:tal2}.
\begin{proof}[of Lemma \ref{lem:tal2}]
By Laplace transform and its discrete analogue, we have
\begin{align*}
 \partial_t^\alpha y_h(t_n)-\bar \partial_\tau^\alpha y_h(t_n)
   &=  \frac{1}{2\pi\mathrm{i}} \int_{\Gamma_{\theta,\delta}^\tau} e^{zt_n} K(z) A_h v_h\,\d z + \frac{1}{2\pi\mathrm{i}} \int_{\Gamma_{\theta,\delta}\setminus\Gamma_{\theta,\delta}^\tau} e^{zt_n} K(z) A_h v_h\,\d z\\
    &:={\rm I}+{\rm II},
\end{align*}
with $K(z)=(\delta_\tau(e^{-z\tau})^{\alpha}-z^\alpha)z^{ -1} (z^\alpha+A_h)^{-1}$.
Recall the following estimate:
\begin{equation}\label{eqn:gen-1}
  | \delta_\tau(e^{-z\tau})^{\alpha}-z^\alpha| \le c \tau z^{1+\alpha},\quad \forall z\in \Gamma_{\theta,\delta}^\tau.
\end{equation}
Then by choosing $\delta=c/t_n$ in the contour $\Gamma_{\theta,
\delta}^\tau$ and the resolvent estimate \eqref{eqn:resol}, we obtain
\begin{align*}
\|{\rm I}\|_{L^2(\Omega)} &\le c \tau \| A_h v_h \|_{L^2(\Omega)} \Big(\int_{\frac{c}{t_n}}^{\frac{\pi\sin\theta}{\tau}} e^{-c\rho t_n} \,\d\rho
   + \int_{-\theta}^\theta ct_n^{-1} \,\d\varphi \Big) \le c\tau t_n^{-1}  \| A v \|_{L^2(\Omega)}.
\end{align*}
Further, by \eqref{eqn:est-kernel}, for any $z=\rho e^{\pm\mathrm{i}\theta}\in \Gamma_{\theta,\delta}\setminus\Gamma_{\theta,\delta}^\tau$ and
choosing $\theta\in(\pi/2,\pi)$ close to $\pi$,
\begin{align*}
  |e^{zt_n} (\delta_\tau(e^{-z\tau})^{\alpha}-z^\alpha)z^{ -1}| & \leq e^{t_n\rho\cos \theta}(c|z|^\alpha e^{\alpha\rho\tau}+|z|^\alpha)|z|^{-1}\leq c|z|^{\alpha-1}e^{-c\rho t_n}.
\end{align*}
Then we deduce
\begin{align*}
\|{\rm II}\|_{L^2\II}
   &\le  c \| A_hv_h\|_{L^2(\Omega)}  \int_{\frac{\pi\sin\theta}{\tau}}^\infty e^{-c\rho t_n} \rho^{-1}\,\d\rho
    \le c\tau t_n^{-1} \| Av\|_{L^2(\Omega)}.
\end{align*}
Thus, we show the assertion for $\beta=0$. For the case $\beta=1$, the identity $A_h(z^\alpha+A_h)^{-1}=I-
z^\alpha(z^\alpha +A_h)$, \eqref{eqn:resol} and \eqref{eqn:gen-1} give
\begin{align*}
\|A_h{\rm I}\|_{L^2\II}
   &\le c \tau \| A_h v_h\|_{L^2(\Omega)} \Big(\int_{\frac{c}{t_n}}^{\frac{\pi\sin\theta}{\tau}} e^{-c\rho t_n}\rho^\alpha \,\d\rho
   + \int_{-\theta}^\theta ct_n^{-1-\alpha} \,\d\varphi \Big)\\
    & \le c\tau t_n^{-1-\alpha}  \| A v\|_{L^2(\Omega)},
\end{align*}
and the bound on $\|A_h{\rm II}\|_{L^2(\Omega)}$ follows analogously, completing the proof for $\beta=1$.
Then the case $\beta\in(0,1)$ follows by interpolation. This shows part (i).
The proof of part (ii) is similar and applies the $L^2(\Omega)$ stability of $P_h$, and
hence the detail is omitted.\qed
\end{proof}

\bibliographystyle{spmpsci}

\end{document}